\documentclass[11pt,a4paper]{amsart}
\usepackage[top=4cm,bottom=4cm,left=3cm,right=3cm]{geometry}
\usepackage[utf8]{inputenc}
\usepackage[T1]{fontenc}
\usepackage[nolist]{acronym}

\usepackage{amsmath, amssymb, amsfonts, amsthm}
\usepackage[all]{xy}
\usepackage{hyperref}
\usepackage{url}
\usepackage{mathtools}
\usepackage{calc}  
\usepackage{enumerate}
\usepackage{enumitem}
\usepackage{float}
\usepackage{bbm}
\usepackage{verbatim}
\usepackage{anyfontsize}

\usepackage{graphicx}
\graphicspath{{images/}{../images/}}
\usepackage{subcaption}
\usepackage{tabularx}
\usepackage{array}
\usepackage{wrapfig}
\usepackage{longtable}
\usepackage{makecell}
\usepackage{multirow, multicol}
\usepackage{ltablex}\keepXColumns

\newcolumntype{Y}{>{\centering\arraybackslash}X}
\newcolumntype{Z}{>{\scriptsize}Y}
\usepackage{rotating}

\usepackage{tikz}
\usetikzlibrary{shapes,arrows,positioning,arrows.meta,chains,positioning,quotes,fit}
\usepackage{blkarray}
\usepackage{tkz-euclide}
\usetikzlibrary{patterns} 
\usepackage{mathrsfs}
\usetikzlibrary{fadings}
\usetikzlibrary{shadows.blur}
\usetikzlibrary{spy}
\usetikzlibrary{decorations.pathreplacing}

\usepackage{nicematrix}

\usepackage{pgfplots}
\usepgfplotslibrary{colorbrewer}
\pgfplotsset{compat=1.17}

\usepackage{subfiles}
\usepackage{blindtext}
\usepackage{chngcntr}
    \usepackage{todonotes}
\usepackage{diagbox}
\usepackage{cleveref}
\usepackage{autonum}
\usepackage{centernot}

\usepackage{algorithm}
\usepackage{algpseudocode}

\usepackage[final]{pdfpages}

\usepackage{appendix}

\newtheorem{theorem}{Theorem}

\newtheorem{lemma}[theorem]{Lemma}
\newtheorem{proposition}[theorem]{Proposition}
\newtheorem{corollary}[theorem]{Corollary}

\theoremstyle{definition}

\newtheorem{definition}[theorem]{Definition}

\theoremstyle{remark}

\newtheorem{example}[theorem]{Example}

\DeclareMathOperator{\image}{Im}

\DeclareMathOperator{\wt}{wt}

\DeclareMathOperator{\trace}{Tr}

\newcommand{\st}{\; | \;} 

\newcommand{\card}[1]{\left\vert \, {#1} \, \right\vert} 
\newcommand{\floor}[1]{\left\lfloor{#1}\right\rfloor} 
\newcommand{\set}[1]{\left\lbrace{#1}\right\rbrace} 

\newcommand{\mini}[1]{\min\left\lbrace #1 \right\rbrace} 

\newcommand{\field}{\mathbb{F}}
\newcommand{\Fpr}{\mathbb{F}_{p^r}}
\newcommand{\Fprn}{\mathbb{F}_{p^r}^n}

\newcommand{\Fqn}{\mathbb{F}_q^n}

\newcommand{\Fp}{\mathbb{F}_p}

\newcommand{\ZZ}{\mathbb{Z}}
\newcommand{\QQ}{\mathbb{Q}}

\newcommand{\ring}{\mathcal{R}}

\newcommand{\Zmod}[1]{\ZZ/{#1}\ZZ}
\newcommand{\units}[1]{(\Zmod{#1})^\times}

\newcommand{\Zps}{\Zmod{p^s}}
\newcommand{\Zpsn}{(\Zps)^n}

\newcommand{\ideal}[1]{\left \langle{#1} \right \rangle}

\newcommand{\code}{\mathcal{C}}

\newcommand{\dualcode}{\code^\perp}

\DeclareMathOperator{\dist}{d}
\newcommand{\weight}{\mathrm{wt}}
\newcommand{\maxWT}{M_{\weight}}

\newcommand{\LW}{\wt_{\mathsf{L}}}

\newcommand{\LD}{\dist_{\mathsf{L}}}

\newcommand{\HW}{\wt_{\mathsf{H}}}

\newcommand{\HomW}{\wt_{\mathsf{hom}}}

\newcommand{\SubW}{\wt_{\lambda}}

\newcommand{\kr}{Krawtchouk }
\newcommand{\partition}{\mathcal{P}}

\newcommand{\decompWq}{\mathbb{D}^{\wt}_{n}}

\newcommand{\enumW}[1]{\mathcal{D}_{#1}^{\wt}}

\newcommand{\enumL}[1]{\mathcal{D}_{#1}^{\mathsf{L}}}
\newcommand{\krawL}{K^{\mathsf{L}}}
\newcommand{\compL}{\mathrm{Comp}^{\mathsf{L}}_{\pi}(\rho)}
\newcommand{\piL}{\pi^{\mathsf{L}}}

\newcommand{\PL}{P^{\mathsf{L}}}
\newcommand{\decompLq}{\mathbb{D}^{\mathsf{L}}_{n}}

\newcommand{\enumH}[1]{\mathcal{D}_{#1}^{\mathsf{hom}}}
\newcommand{\krawH}{K^{\mathsf{hom}}}
\newcommand{\compH}{\mathrm{Comp}^{\mathsf{hom}}_{\pi}(\rho)}

\newcommand{\PHsym}{\partition_{\mathrm{sym},n}^{\mathsf{hom}}}
\newcommand{\PH}{\partition^{\mathsf{hom}}}
\newcommand{\blockH}{P^{\mathsf{hom}}}

\newcommand{\decompHq}{\mathbb{D}^{\mathsf{hom}}_{n}}
\newcommand{\ZUSRset}{\set{Z,U,S,R}}

\newcommand{\enumS}[1]{\mathcal{D}_{#1}^{\lambda}}
\newcommand{\krawS}{K^{\lambda}}
\newcommand{\compS}{\mathrm{Comp}^{\lambda}_{\pi}(\rho)}

\newcommand{\PS}{\partition^{\lambda}}
\newcommand{\PSsym}{\partition_{\mathrm{sym},n}^{\lambda}}
\newcommand{\blockS}{P^{\lambda}}

\newcommand{\decompSq}{\mathbb{D}^{\lambda}_{n}}
\newcommand{\IS}{I^{\lambda}}

\newcommand{\wwe}[1]{\mathsf{wwe}_{#1}}
\newcommand{\Lwe}[1]{\mathsf{Lwe}_{#1}}
\newcommand{\Homwe}[1]{\mathsf{Homwe}_{#1}}

\newcommand{\Swe}[1]{\lambda\mathsf{we}_{#1}}
\newcommand{\SDe}[1]{\lambda\mathsf{dec}_{#1}}

\newcommand{\wtdec}[1]{\pi^{(#1)}}

\definecolor{dark_red}{RGB}{150,0,0}
\definecolor{dark_green}{RGB}{0,150,0}
\definecolor{dark_blue}{RGB}{0,0,150}
\definecolor{dark_pink}{RGB}{80,120,90}
\definecolor{myteal}{RGB}{0,128,106}
\definecolor{myblue}{RGB}{69, 220, 180}
\definecolor{myred}{RGB}{222, 0, 49}
\definecolor{myyellow}{RGB}{255,170,0}

\colorlet{teall}{teal!75!yellow}

\colorlet{review}{blue!80!cyan}

\title[]{Coarsest Fourier-reflexive Partitions for the Lee, Homogeneous and Subfield Metric}

\begin{document}

\author[J. Bariffi]{Jessica Bariffi$^{1}$}
\author[G. Cavicchioni]{Giulia Cavicchioni$^{2}$}
\author[V. Weger]{Violetta Weger$^1$}

\address{$^1$Department of Electrical and Computer Engineering\\
        Technical University of Munich\\
        Germany 
}
\email{jessica.bariffi@tum.de}

\address{$^2$Institute of Communication and Navigation \\ 
        German Aerospace Center \\ 
        Germany
	}
\email{giulia.cavicchioni@unitn.it}

\email{violetta.weger@tum.de}

\begin{abstract}
MacWilliams identities relate the weight enumerators of a code with those of its dual and are classically formulated with respect to the Hamming weight. For other metrics, however, these identities often fail when considering the  weight partition of the ambient space. It is known that MacWilliams identities hold for  enumerators associated with Fourier-reflexive partitions, and that orbits of subgroups of the linear isometry group always yield such partitions. 
This raises the question whether, for metrics beyond the Hamming metric, there exist meaningful partitions that lie strictly between the weight partition and the orbit partition: finer than the latter, yet still coarse enough to retain useful MacWilliams-type identities.
In this work, we study this question for finite chain rings endowed with additive metrics. For the Lee metric, we show that the partition induced by the action of the full group of linear isometries is already the coarsest Fourier-reflexive partition refining the weight partition. 
In particular, no intermediate partition exists that is both finer than the Lee weight partition and Fourier-reflexive. 
We refer to this partition as the Lee partition and show that it allows the recovery of all additive weight enumerators over the ring.
In contrast, for the homogeneous metric and for the subfield metric, we identify new, significantly coarser symmetrized partitions that remain Fourier-reflexive and still allow the recovery of the corresponding weight enumerators. We prove that these partitions are the coarsest such symmetrized partitions for which MacWilliams-type identities hold. As an application, we derive linear programming bounds based on the resulting MacWilliams identities.
\end{abstract}
\maketitle


\section{Introduction}\label{sec:intro}
The MacWilliams identities have been a major breakthrough in understanding the theory of linear codes and mark one of the most celebrated results. In a nutshell, the MacWilliams identities state that the weight enumerator of a code is related to, and completely determines, the weight enumerator of the dual code. 
The original result has been proven in \cite{macwilliams1963theorem} by Jessie MacWilliams and several generalizations and reformulations, through association schemes, generating functions or graphs, have appeared since \cite{delsarte1973algebraic,klemm1987identitat,klemm1989selbstduale,wood1999duality, honold2001macwilliams}. Other generalizations of the MacWilliams identities have been discussed in \cite{simonis1995macwilliams} and further extensions to different code families, like convolutional codes, quantum codes and others, followed \cite{byrne2007linear,gluesing2008macwilliams,gluesing2009macwilliams,bocharova2010weight,yildiz2014linear}. 
The MacWilliams identities also allow to state a Linear Programming (LP) bound, giving the tightest known upper bound on the size of a code with prescribed minimum distance and length. For an overview we refer the interested reader to \cite{delsarte1998association}.

This powerful original result restricts itself to classical codes over finite fields endowed with the Hamming metric. However, as soon as we change the metric, the situation becomes more complex. While the MacWilliams identities are still valid for codes over $\Zmod{4}$ equipped with the Lee metric \cite{hammons1994z}, in 2015, Shi et al. \cite{shi2015note} and \cite{wood2015some} proved their failure over $\Zmod{8}$. Moreover, Abdelghany and Wood  \cite{abdelghany2020failure} showed the nonexistence of any version of the MacWilliams identities for Lee weight enumerators over $\Zmod{m}$ for any positive integer $m \geq 5$. 
A similar result holds for the homogeneous weight as well: apart from the trivial cases $\Zmod{p}$, for a prime $p$, where the homogeneous metric coincides with the Hamming metric, the homogeneous MacWilliams identities fail \cite{gluesing2015fourier,wood2023homogeneous}. In \cite{wood2023weights,wood2026weights}, the rarity of weights that satisfy the MacWilliams identities is observed.
Changing the partition of a code and, more precisely, moving away from enumerating codewords of the same weight, has already been suggested by MacWilliams herself for Lee-metric codes over finite fields \cite{macwilliams1972macwilliams}. Similarly, this idea has been supported for Lee-metric codes using association schemes in \cite{astola1982leescheme,patrick1986lee} where the existence of MacWilliams-type identities have been shown. 
The study of partitioning the ambient space according to a pre-defined property dates back to Zinoviev and Ericson \cite{zinoviev1996fourier} in 1996. They studied additive codes and analyzed for which enumerators of the code and its dual  MacWilliams-type identities are fulfilled, which laid the foundation of \textit{Fourier-reflexive} partitions. In \cite{gluesing2015fourier}, Gluessing-Luerssen, studied partitions of additive codes over finite Frobenius rings and showed that MacWilliams identities exist if the partition of the underlying ambient space is Fourier-reflexive. In a follow-up work \cite{Heide}, she proved that partitioning a code in the homogeneous metric into the codewords of the same homogeneous weight  yields to a failure of the MacWilliams identities. Similarly, for codes in the sum-rank metric, \cite{byrne2021fundamental} showed the failure of the classical MacWilliams identities and the existence of  MacWilliams-type identities using a finer partition. 

In this paper, we focus on codes over a finite chain ring $\ring$ endowed with an additive weight (i.e., the weight of a codeword is given by the sum of the weights of its entries over $\ring$).
For additive weights on codes 
$\code \subseteq \ring^n$, it is natural to consider partitions of the ambient space 
$\ring$ that induce symmetrized partitions of 
$\ring^n$. MacWilliams-type identities for such symmetrized partition enumerators were studied in \cite{gluesing2015fourier}, where it was shown that they hold whenever the underlying partition of $\ring$ is Fourier-reflexive. Classical examples include the Hamming weight partition, the singleton partition, and partitions induced by the orbits of subgroups of the additive automorphism group.
These examples illustrate a fundamental trade-off. The Hamming weight partition is very coarse and satisfies MacWilliams identities, but it does not allow to recover other additive weight enumerators. At the opposite extreme, the singleton partition is Fourier-reflexive and allows the recovery of all additive weights, but it is maximally fine and becomes computationally infeasible as the alphabet size grows. Orbit partitions arising from isometry groups lie between these two extremes and are guaranteed to be Fourier-reflexive, but they are often still unnecessarily fine.
The central goal of this work is therefore to identify, for a given additive metric, the coarsest symmetrized partition that is Fourier-reflexive and still determines the corresponding weight enumerator. We address this problem for finite chain rings endowed with the Lee metric, the homogeneous metric, and the subfield metric. In particular, we show that for the Lee metric no such intermediate partition exists: the orbit partition induced by the full isometry group is already the coarsest Fourier-reflexive refinement. In contrast, for the homogeneous and subfield metric we construct new, significantly coarser symmetrized partitions that retain MacWilliams-type identities while still allowing the recovery of the respective weight enumerators.
The paper is organized as follows. Section \ref{sec:preliminaries} will serve as introduction and recap of the main definitions and results needed throughout the paper. We start with the notion of a linear code over a finite chain ring endowed with an additive weight (where we specifically introduce the Hamming, Lee, homogeneous and subfield weight). We then recap the theory of group characters stating some of their main properties, such as the Schur orthogonality. The section is completed by introducing the MacWilliams identities in their original form together with the Krawtchouk coefficients. Instead of partitioning tuples by weight, in Section \ref{sec:new_macwilliam} we present novel Fourier-reflexive partitions with respect to the Lee, homogeneous and subfield weight.
As a first step we consider a finite chain ring endowed with any additive metric and partition the code into codewords of the same (Lee weight) decomposition which we refer to as the \textit{Lee partition} of the code. 
We show that the Lee decomposition enumerator of a code is fully determined by the Lee decomposition enumerator of its dual code, i.e.,  MacWilliams-type identities exist for any additive weight over any finite chain ring, and, in particular, for codes in the Lee metric. Even though the result holds for any additive metric over a finite chain ring, for some metrics there are coarser and more natural partitions to choose. Thus, in a next step we introduce partitions for the homogeneous weight and the subfield weight, respectively, that are more suited and coarser   and still allow for   MacWilliams-type identities. To complete the study, we derive linear programming bounds for the Lee, homogeneous and subfield metric in Section \ref{sec:lpbound}. Finally, concluding remarks are given in Section \ref{sec:conclusion}.
\section{Preliminaries}\label{sec:preliminaries}
In this section we introduce the main definitions and basic results needed throughout this paper. We denote by $p$ a prime number.
Let $\ring$ be a finite chain ring with maximal ideal generated by $\gamma$, let $s$ be its nilpotency index and $q=p^r$ be the size of the residue field, for a positive integer $r$. While the main result holds for any finite chain ring, we often focus on integer residue rings. 
We denote by $p$ a prime number and for a positive integer $s$ we let $\Zps$ denote the finite integer residue ring, and we denote by $\units{p^s}$ the set of units modulo $p^s$. 
Note that a finite integer residue ring $ \Zps$ is a finite chain ring with $q=p$, $\gamma=p, r=1$ and nilpotency index $s$. 
Given a positive integer $n$ and an $\ell$-tuple of nonnegative integers $k :=(k_1, \ldots, k_{\ell})$ satisfying that $\sum_{i=1}^\ell k_i= n$, we denote by
\begin{align}
    \binom{n}{k} := \binom{n}{k_1, \ldots, k_{\ell}} = \frac{n!}{k_1! \ldots k_{\ell}!}
\end{align}
the multinomial coefficient. Moreover, given a statement $A$, we denote by $\mathbbm{1}_A$ its indicator function, that is,
\begin{align}
    \mathbbm{1}_A = 
    \begin{cases}
        1 & A \text{ is true,}\\
        0 & \text{otherwise.}
    \end{cases}
\end{align}
Furthermore, for a positive integer $m$, we denote by $I_m$ the $m\times m$ identity matrix.

\subsection{Codes and Additive Weights over Finite Chain Rings}\label{subs:finitechain}
In this section we  introduce linear codes over finite chain rings and their main properties. In this paper, we exclusively focus on commutative chain rings in order not to distinguish between left and right ideals.

\begin{definition}
    Let $\ring$ be a finite chain ring. An \emph{$\ring$-linear code} $\code$ of length $n$ is an $\ring$-submodule of $\ring^n$. The free module $\ring^n$ is called the \emph{ambient space} of $\code$ and the elements of $\code$ are called \emph{codewords}.
\end{definition}

The fundamental theorem of finite abelian groups implies that any $\ring$-module (i.e., any $\ring$-linear code $\code$) is isomorphic to the following direct sum of $\ring$-modules
\begin{align}
    \code \cong (\ring / \gamma^s\ring)^{k_0} \times (\ring / \gamma^{s-1}\ring)^{k_1} \times \cdots \times (\ring /\gamma\ring)^{k_{s-1}} .
\end{align}
The unique $s$-tuple $(k_0, k_1, \ldots, k_{s-1})$ is called the \emph{subtype} of $\code$ and $k_0$ is called its \emph{free rank}.\medbreak

A code $\code \subseteq \ring^n$ can be represented by a \emph{generating set}, that is a subset of codewords that generates $\code$ as an $\ring$-submodule. We call a generating set a \emph{minimal generating set} if it is minimal with respect to inclusion. The cardinality of a minimal generating set of $\code$ is called the \emph{rank} $K$ and can be computed as $K = \sum_{i = 0}^{s-1}k_i$. 

We  endow the ambient space $\ring^n$ with the standard inner product, i.e., for $x, y \in \ring^n$ we set 
\begin{align}
    \ideal{x, y} := x y^\top = \sum_{i = 1}^n x_i y_i.
\end{align}
\begin{definition}
    Given a linear code $\code \subseteq \ring^n$, the \emph{dual} $\dualcode$ of $\code$ is defined as
    \begin{align}
        \dualcode := \set{x \in \ring^n \st \ideal{x, c} = 0 \text{ for every } c \in \code}.
    \end{align}
\end{definition}
In ring-linear coding theory, the notion of duality is in general not well-defined. However, in \cite{wood1999duality}, Wood proved that the dual of a code over a Frobenius ring, and hence over a finite chain ring, is well-defined. Thus, it holds that $(\code^\perp)^\perp=\code$.
It is straightforward to see that a parity-check matrix $H$ of $\code$ is a generator matrix of $\dualcode$. Moreover, if $\code $ is of rank $K$ and subtype $(k_0,\dots,k_{s-1})$ then $\dualcode$ has rank $n-k_0$ and subtype $(n-K,k_{s-1},\dots,k_1)$.\medskip

One of the main goals of coding theory is to correct errors that occur during transmission in a received message. A linear code is capable of correcting a certain amount of errors which are measured using a given metric induced by a weight function over the ring $\ring$.

\begin{definition}
    Let $\ring$ denote a finite chain ring. A \emph{weight} over $\ring$ is a function $\weight\colon  \ring \to \QQ $ satisfying 
    \begin{enumerate}
        \item  $\weight(0)=0 $ and $\weight (x)>0 $ for all $x\ne 0$;
        \item $\weight(x)=\weight(-x)$;
        \item  $\weight(x+y)\le \weight(x)+\weight(y)$.
    \end{enumerate}
\end{definition}
A weight over a finite chain ring $\ring$ naturally induces a \emph{distance} which is a function $ \mathrm d\colon \ring\times\ring \to \QQ $ defined by $\mathrm{d}(x,y):=\weight(x-y)$.
The coordinate-wise extensions of the weight and distance function will also be denoted by $\weight$ and $\mathrm{d}$ which we call \emph{(additive) weight} and \emph{(additive) distance}, respectively. In particular, given $x\in \ring^n$, we define
\begin{align}
    \weight(x):=\sum_{i=1}^n \weight(x_i).
\end{align}
It is easy to see that such an additive distance over $\ring$ satisfies the properties of a metric. We will now introduce the main additive weight functions which are of interest in the course of this paper. Each of the weights over $\ring$ induce a distance function as described above.
The most prominent weight in coding theory is the Hamming weight.
\begin{definition}\label{def:Hamming}
    Let $\ring$ be a finite chain ring. The \emph{Hamming weight} of an $n$-tuple $x \in \ring^n$ is given by the cardinality of the support of $x$, i.e.,
    \begin{align}
        \HW(x) := \card{\set{i \in \{ 1, \ldots , n\} \st x_i \neq 0}}.
    \end{align}
\end{definition}

We now focus on the case where $\ring = \Zps$ is a finite integer residue ring modulo $p^s$. Hence, we consider codes to be submodules of $(\ZZ/ p^s \ZZ)^n$. 
The Lee weight, introduced in 1958 by Lee \cite{lee1958some}, is a generalization of the Hamming weight over the binary.

\begin{definition}
    The \emph{Lee weight of an element} $a \in \Zps$ is given by
    \begin{align}
        \LW(a) := \mini{a, \card{p^s-a}}.
    \end{align}
    Similarly, we define the \emph{Lee weight of an $n$-tuple} $x \in \Zpsn$ additively by
    \begin{align}
        \LW(x) := \sum_{i = 1}^n \LW(x_i).
    \end{align}
\end{definition}
The following upper and lower bounds hold trivially  for any $a \in \Zps$ and $x \in \Zpsn$
\begin{align}
    0 \leq \LW(a) \leq \floor{p^s/2} \quad \text{and} \quad \HW(x)\leq \LW(x) \leq \floor{p^s/2}\HW(x).
\end{align}

In 1997, Constantinescu and Heise first introduced the homogeneous weight for integer residue rings \cite{constantinescu1997metric}, which was later extended  to finite rings in \cite{greferath2000finite} and \cite{nechaev1999weighted}. For this we consider again a finite chain ring $\ring$ with socle $\mathcal{S}.$ Recall that the socle $\mathcal{S}$ of a ring is defined as the sum of the minimal nonzero submodules. For example, for $\ring= \Zps$ we have $\mathcal{S}= \langle p^{s-1} \rangle$. More generally, if the maximal ideal of $\ring$ is $\langle \gamma \rangle$, we have $\mathcal{S} = \langle \gamma^{s-1} \rangle.$
The homogeneous weight extends the Hamming weight coinciding with it for finite fields and with the Lee weight for $\Zmod{4}$. In the following we define the homogeneous weight in its normalized form. 

\begin{definition}\label{def:hom_metric}
    Let $\ring$ be a finite chain ring with residue field of size $q$. The \emph{homogeneous weight} of $a \in \ring$ is defined as
    \begin{align}
        \HomW(a)
            =
        \begin{cases}
            0 & \text{if } a = 0,\\
            1 & \text{if } a \not\in \mathcal{S},\\
            \frac{q}{q-1} & \text{if } a \in \mathcal{S} \setminus \{0\}.
        \end{cases}
    \end{align}
    The  homogeneous weight extends additively to tuples $x \in \ring^n$, i.e., 
    \begin{align}
        \HomW(x)=\sum_{i=1}^n \HomW(x_i).
    \end{align} 
\end{definition}

Lastly, let us consider finite fields, more precisely finite extension fields $\field_{p^r}$. In \cite{grassl2022subfield}, the authors introduced the \emph{subfield metric} over  $\field_{p^r}$ for quantum error correcting codes. This metric is defined similarly to the homogeneous metric over finite chain rings. Namely, it gives weight one to nonzero elements in the underlying base field and all other nonzero elements in $\field_{p^r}$ are of weight $\lambda \geq 1$.
\begin{definition}\label{def:subfield_metric}
    Let $a \in \field_{p^r}$ and $\lambda \geq 1$. The $\lambda$-\emph{subfield weight} of $a$ is defined as
    \begin{align}
        \SubW(a)
            =
        \begin{cases}
            0 & \text{if } a = 0,\\
            1 & \text{if } a \in \Fp^{\times},\\
            \lambda & \text{if } a \in \field_{p^r}\setminus \Fp.
        \end{cases}
    \end{align}
    The $\lambda$-subfield weight extends additively to vectors $x \in \field_{p^r}^n$, i.e., 
    \begin{align}
        \SubW(x) = \sum_{i=1}^n \SubW(x_i).
    \end{align}
\end{definition}

\subsection{Group Characters}
We now recall the required basics on group characters. In a nutshell, a group character defines a representation of a group $G$ in terms of a complex function. Working over cyclic groups, characters are used in combination with the discrete Fourier transform. Throughout this section, let $\mathbb C^*\coloneqq \set{z \in \mathbb{C} \st \card{z} = 1}$. 

\begin{definition}\label{def:character}
    Let $G$ be a  finite abelian group. A \emph{character} $\chi$ of $G$ is a complex-valued map $\chi: G \longrightarrow \mathbb C^*$, such that for every $\alpha, \beta \in G$ it holds
    \begin{align}
        \chi(\alpha + \beta) = \chi(\alpha)\chi(\beta).
    \end{align}
\end{definition}
We denote by \( \widehat{G} \) the set of all characters of \( G \), where  
$\widehat{G} = \operatorname{Hom}_{\mathbb{Z}}(G, \mathbb{C}^*)$  
forms a group along with the addition
\begin{align}
    (\chi_i+ \chi_j)(g) = \chi_i(g) \chi_j(g), \quad \text{for all } \chi_i, \chi_j \in \widehat{G} \text{ and } g \in G.
\end{align}  
This group is called the \emph{character group} of \( G \). The identity element of \( \widehat{G} \) is the \emph{principal character} \( \chi_0 \), given by \( \chi_0(g) = 1 \) for all \( g \in G \).  
It is well known that \( G \) and \( \widehat{G} \) are isomorphic, though not naturally so \cite{terras1999fourier}. Consequently, we have \( |G| = |\widehat{G}| \). Moreover, \( G \) is naturally isomorphic to the double character group \( \widehat{\widehat{G}} \).

\begin{example}\label{ex:char} 
    Let $\ring$ be a finite chain ring of characteristic $p^s$. It is well-known that there exists a subring $\mathcal{T} \subseteq \ring$  such that $\ring$ is isomorphic to a direct product of multiple copies of $\mathcal{T}$  as  abelian groups \cite{clark1973finite}. Therefore $\ring= \prod_{i=1}^t\mathcal{T}$, where $\mathcal{T}$ is in particular a Galois ring and it is unique up to isomorphism. Given $\xi$ a $p^s$-th root of unity, for any $a=(a_1,\dots,a_t),\ x=(x_1,\dots,x_t)\in \ring$, a character $\chi_a(x)$ of $\ring$ has the form
    \begin{align}
       \chi_a(x)=\prod_{i=1}^t\xi^{\trace(a_ix_i)}=\xi^{\trace(a_1x_1+\dots+a_tx_t)}= \xi^{\trace(\ideal{a,x})} \ . 
    \end{align}
    We can derive expressions for characters in $\ring^n$. Given  $a = (a^{(1)},\ldots, a^{(n)}), x=(x^{(1)},\ldots, x^{(n)})\in \ring^n,$ where $a^{(i)}=(a^{(i)}_1, \ldots, a^{(i)}_t) \in \ring$ all the characters are of the form 
    \begin{align}
        \chi_a(x)=\prod_{i=1}^n\xi^{\trace\left(\ideal{a^{(i)},x^{(i)}}\right)} \ ,
    \end{align}
    where $\xi$ is a $p^s$-th root of unity. 
\end{example}

An important property of characters is the Schur orthogonality. Given a character $\chi$ of a group $G$, the Schur orthogonality gives an explicit formula for the sum of the evaluations of $\chi$ in every element of the group $G$.
\begin{lemma}[Schur Orthogonality over a Group]\label{lemma:schur_group}
    Let $G$ be a finite group and $\chi$ a character of $G$. Then it holds that
    \begin{align}
        \sum_{g \in G} \chi(g) =
        \begin{cases}
            \card{G} & \text{ if } \chi \text{ is the principal character,}\\
            0 & \text{ otherwise}.
        \end{cases}
    \end{align}
\end{lemma}
\begin{proof}
    If $\chi$ is a principal character, we have $\chi(g) \equiv 1$ for every $g \in  G$ and therefore $\sum_{g \in G} \chi(g) = \card{G}$.
    Thus, let us assume that $\chi$ is a non-principal character of $G$. This means that there exists an element $h \in G$ such that $\chi(h) \neq 1$. Since $G$ is a group, we have
    \begin{align}
        \sum_{g \in G} \chi(g) = \sum_{g \in G} \chi(g + h) = \sum_{g \in G} \chi(g)\chi(h) = \chi(h)\sum_{g \in G} \chi(g),
    \end{align}
    which is equivalent to $(1-\chi(h))\sum_{g \in G} \chi(g) = 0$. As $\chi(h) \neq 1$ it must hold that $\sum_{g \in G} \chi(g) = 0$.
\end{proof}
A similar result exists for subgroups (or linear codes).
However, in that case, we distinguish between characters whose kernel is contained in the given subgroup and those whose kernel is not (see for instance \cite{wood1999duality}). More formally, let $H \subset G$ be a subgroup of a finite abelian group $G$. Define the \emph{annihilator} of $H \subset G$ as
\begin{align}\label{eq:annHinG}
    (\widehat{G} : H) := \set{\chi \in \widehat{G} \st \chi(h) = 1 \text{ for all } h \in H}.
\end{align}
\begin{definition}
   Let $\chi$ be a character of $G$ and let $H$ be a subgroup of $G$. Then $\chi $ is said to be \textit{generating for $H$}, or simply \emph{generating}, if $\chi \not\in (\widehat{G} : H)$, that is, $H\not\subset \ker( \chi) $. 
\end{definition} The orthogonality relation in Lemma \ref{lemma:schur_group} extends to a subgroup as follows.
\begin{lemma}[Schur Orthogonality over a Subgroup]\label{lemma:schur_code}
    Let $G$ be a finite abelian group and $H \subset G$ be a subgroup. Then the following Schur orthogonality property holds
    \begin{align}\label{eq:schur}
        \sum_{h \in H} \chi(h)
            =
        \begin{cases}
           0 & \text{ if } \chi \text{ is  generating for $H$,} \\
             \card{H} & \text{ otherwise.}
        \end{cases}
    \end{align}
\end{lemma}

As a consequence of Schur's Lemma \ref{lemma:schur_group}, we highlight two relevant properties of characters that will be extensively used to prove MacWilliams-type identities.
\begin{corollary}\label{cor:charactersum_ideal}
    Let $G$ be a finite abelian group and let $H$ be a subgroup of $G$. If $\chi$ is a   {generating} character for $H$ (i.e., $\chi \not\in (\widehat{G}:H)$) it holds that
    \begin{itemize}
        \item $\sum_{x \in H} \chi(x)=0$;
        \item $\sum_{ x \in H  \setminus\{0\}} \chi(x) = -1$. 
    \end{itemize}
    In particular, if $\ring$ is a finite chain ring and $I$ is an ideal in $\ring$,  for a generating character $\chi$ over $\ring$ we have $\sum_{x \in I} \chi(x)=0$, while $\sum_{ x \in I \setminus\{0\}} \chi(x) = -1.$ 
\end{corollary}

\subsection{Partitions and Induced Partitions}
We now turn our attention to partitions of a finite abelian group $G$.
Following \cite{zinoviev2009fourier}, a partition $\mathcal{P} = P_1\mid P_2\mid \dots\mid P_m$ of $G$ consists of non-empty pairwise disjoint sets $P_i$, $1\le i \le m$, that cover $G$. The sets of a given partition are called \emph{blocks}.
\begin{definition}\label{def:coarser-partition}
    Let $\mathcal{P}$ and $\mathcal{Q}$ be two partitions of $G$. The partition  $\mathcal{P}$ is called \emph{finer} than $\mathcal{Q}$ if the blocks of $\mathcal{P}$ are contained in the blocks of $\mathcal{Q}$. We also say that $\mathcal Q$ is \emph{coarser} than $\mathcal{P}$ and we write $\mathcal{P} \leq \mathcal{Q}$.
\end{definition}
The partition of a group $G$ can be extended to a partition of the $n$-fold cartesian product of the group $G$ as follows \cite{byrne2007linear,gluesing2015fourier}.
\begin{definition}\label{def:symmetrized-partition}
    Let $G$ be a finite abelian, nontrivial group and consider a partition $\mathcal{P} = P_1 \st \cdots \st P_m$ of $G$. For $g, g' \in G^n$, define the equivalence relation $g \sim g'$ if $|\set{i \in \{1, \ldots, n \} \st g_i \in P_j}| = |\set{i \in \{ 1, \ldots, n \} \st g'_i \in P_j}|$ for all $j \in \{ 1, \ldots, m\}$. The resulting partition $\mathcal{P}_{ \mathrm{sym},n}$ on $G^n$ is called the \textit{induced symmetrized partition}.
\end{definition}

We now define a partition of the set of characters $\widehat{G}$ of a finite abelian group which we refer to as the \textit{dual partition}.
\begin{definition}\label{def:dual-partition}
    Let $\mathcal{P} = P_1\mid P_2\mid \dots\mid P_m$ be a partition of a finite abelian group $G$. Let $\widehat{\mathcal{P}} = Q_1\mid Q_2\mid \dots\mid Q_\ell$ denote the partition of $\widehat{G}$ where $\chi, \chi' \in \widehat{G}$ lie in the same block $Q_j$, for $j \in \{1, \ldots, \ell\}$, if and only if, for every $i =\in \{1, \ldots, m\}$, it holds
    \begin{align}
        \sum_{x \in P_i} \chi(x) = \sum_{x \in P_i} \chi'(x) .
    \end{align}
    We call $\widehat{\mathcal{P}}$ the \textit{dual partition} of $G$. Furthermore, a partition $\mathcal{P}$ is called \textit{self-dual} if $\mathcal{P} = \widehat{\mathcal{P}}$.
\end{definition}
In \cite{gluesing2015fourier}, Gluesing-Luerssen referred to a partition $\mathcal{P}$ as \textit{Fourier-reflexive} if $\widehat{\!\widehat{\mathcal{P}}} = \mathcal{P}$, and showed that a partition $\mathcal{P}$ is Fourier-reflexive if and only if $\mathcal{P}$ and $\widehat{\mathcal{P}}$ consist of the same number of blocks. This definition aligns with the \textit{Fourier-invariant} partitions introduced in \cite{zinoviev1996fourier}. Additionally, \cite[Theorem 3.3]{gluesing2015fourier} shows that an induced symmetrized partition $\mathcal{P}_{\mathrm{sym},n}$ is Fourier-reflexive if $\mathcal{P}$ is. \medbreak

In the following, we focus on self-dual partitions which, by definition, are Fourier-reflexive.
\begin{definition} \label{def:kr}
    Let $\mathcal{P} = P_1\mid P_2\mid \dots\mid P_m$ and $\mathcal{Q} = Q_1\mid Q_2\mid \dots\mid Q_m$ be two partitions of the finite abelian group $G$.  For $a \in Q_j$, let $\chi_a$ be a character of $G$. If the sum $\sum_{x\in P_i}\chi_a(x)$ does not depend on the choice of $a\in Q_j$, we define the \emph{\kr coefficient of $i$ with respect to $j$}  as 
  \begin{align}\label{eq:kr}
       K_j(i)=\sum_{x\in P_i}\chi_a(x) \ \text{ for any } a \in P_j \ .
  \end{align}
\end{definition}

Identifying the group $G$ with its character group $\widehat{G}$,  the partition $\mathcal Q$ induces a partition on $\widehat{G}$. Zinoviev and Ericson \cite[Lemma 1]{zinoviev1996fourier} showed that the \kr  coefficients are well-defined if and only if the partitions are Fourier-invariant. Furthermore, the existence of the \kr coefficients depends on the partition of the group rather than the group itself.
In coding theory, partitions are often induced by a weight function on the ambient space. Given a finite chain ring $\ring$ and an additive weight $\wt$ on $\ring^n$, we can partition $\ring^n$ into tuples of the same weight. That is, we consider a partition
$\mathcal{P}^{\wt} = P^{\wt}_0 \st \cdots \st P^{\wt}_{\card{\image(\wt)} - 1}$ where
\begin{align}
    P^{\wt}_i  \coloneqq \{x\in \ring^n \st \wt(x)=i \}.
\end{align}

\subsection{MacWilliams Identities over Finite Chain Rings}\label{subsec:MWIdentities}
We finally recap the MacWilliams identities for the Hamming metric and the non-existence results in the Lee, homogeneous in this section.

Again, let $\ring$ be a finite chain ring and let $\weight\colon \ring^n \longrightarrow \QQ$ be an additive weight over $\ring$. Furthermore, we consider a linear code $\code \subseteq \ring^n$. From now on, we will denote by $\maxWT$ the maximum value of $\weight$ on $\ring$.
The \emph{$\weight$-weight distribution} of $\code$ is a tuple which specifies the number of codewords of each possible weight $i \in \set{0,\ldots ,  nM_{\wt}}$, i.e.,
\begin{align}\label{equ:weight_enumerator}
    A^{\wt}(\code) := \big(A_0^\weight(\code),A_1^\weight(\code),\dots, A_{n\maxWT}^\weight(\code)\big), 
\end{align}
where $A_i^\weight(\code) := \card{\set{c \in \code \st \wt(c) = i}}$. If $\weight$ and $\code$ are clear from the context, we will simply write $A_i $ instead of $A_i^\weight(\code) $. Moreover, the weight distribution determines the $\weight$-\emph{weight enumerator} of $\code$ defined as
\begin{align}
   \wwe{\code} = \sum_{i=0}^{n\maxWT}A_i^\weight(\code) X^{n\maxWT-i}Y^i \ . 
\end{align}
That is, for instance, the Lee weight enumerator of a code $\code$ is denoted by $\Lwe{\code}$, the homogeneous weight enumerator by $\Homwe{\code}$ and the subfield weight enumerator by $\Swe{\code}$. 
Although the weight distribution does not entirely characterize a code, it provides significant information. The  research on this topic dates back to MacWilliams' work  \cite{macwilliams1963theorem}, when she proved the famous  \emph{MacWilliams identities} \eqref{eq:Macwilliams_Ham}. The identities establish a relation between the Hamming weight enumerators of a linear code over a finite field and its dual.
\begin{theorem}[MacWilliams Identities \text{\cite[Lemma 2.2]{macwilliams1963theorem}}]
    Let $\code \subseteq \mathbb{F}_q^n$ be a linear code over $\field_q$ and let $\dualcode$ be its dual code. Let $(A_i)_{i=0,\dots,n}$ and $(B_i)_{i=0,\dots,n}$ be the Hamming weight distributions of $\code$ and $\dualcode$, respectively. For any $ 0\le \nu\le n$ it holds
    \begin{equation}\label{eq:Macwilliams_Ham}
        \sum_{j=0}^{n-\nu} \binom{n-j}{\nu} A_j=q^{n-\nu}\sum_{j=0}^\nu \binom{n-j}{n-\nu}B_j \ .
    \end{equation} 
\end{theorem}

There are several techniques to prove the MacWilliams identities each resulting in a distinct set of equations, but equivalent to the MacWilliams identities. One of the most widely used methods involves the Poisson summation formula \cite{assmus1974coding}.
On the other hand, Delsarte proposed an interesting approach based on association schemes and \kr coefficients \cite{Delsarte1972BoundsFU}. 

\begin{theorem}\label{thm:macwilliam_krawtchouk}
Let $\code \subseteq \mathbb{F}_q^n$ be a linear code and let  $(A_i)_{i=0,\dots,n}$ and $(B_i)_{i=0,\dots,n}$ be the Hamming weight distributions of $\code$ and $\dualcode$, respectively. For any  $ 0\le j \le n$,
\begin{align}
    B_j=\frac{1}{|\code|} \sum_{i=0}^n A_i K_j(i),  
\end{align}
where the $\kr$ coefficients are given by $K_j(i) =\sum_{k=0}^j (-1)^k(q-1)^{j-k} \binom{i}{k}\binom{n-i}{j-k}$.
\end{theorem}
For the partition of $\Zpsn$ into tuples of given Lee weight, however, the MacWilliams identities fail, whenever $p^s \geq 5$.
\begin{theorem}[\text{\cite[Theorem 1.1]{abdelghany2020failure}}]\label{thm:nonexistence_Macwill_Lee}
  For any $p^s\ge 5$ there exist linear codes $\code_1,\code_2$ over $\Zps$ satisfying $\Lwe{\code_1}= \Lwe{\code_2} $ and $ \Lwe{\dualcode_1}\ne \Lwe{\dualcode_2}$. 
\end{theorem}
Note that Theorem \ref{thm:nonexistence_Macwill_Lee} has been shown for any $\Zmod{m}$ with $m \geq 5$. 
A similar result holds for the homogeneous weight as well.

\begin{theorem}[\text{\cite[Theorem 3.2]{wood2023homogeneous}}]\label{thm:nonexistence_Macwill_Hom}
    For any prime number $p$ and any positive integer $s\ge 2$ with $p^s > 5$, there exist linear codes $\code_1,\code_2$ over $\Zps$ satisfying that $\Homwe{\code_1}= \Homwe{\code_2} $ and $ \Homwe{\dualcode_1}\ne \Homwe{\dualcode_2}$.
\end{theorem}

Both, the Lee and homogeneous weight enumerator, can be interpreted as symmetrized weight partitions on $\ring^n$ induced by the weight partition on $\ring$ and are not Fourier-reflexive. This motivates the need to consider a different partition of the ambient space.

\section{Partition-based MacWilliams Identities}\label{sec:new_macwilliam}
As discussed in the previous section, the MacWilliams identities generally fail for some weight enumerators like, for instance, the Lee weight enumerator and the homogeneous weight enumerator (see Theorems \ref{thm:nonexistence_Macwill_Lee} and \ref{thm:nonexistence_Macwill_Hom}, respectively). However, using partitions different from the weight partition yields MacWilliams-\emph{type} identities provided that the partition is Fourier-reflexive, as established in \cite{gluesing2015fourier}. In this section we make use of these results and present for different additive metrics over $\ring$ each a symmetrized partition that is Fourier-reflexive (and hence satisfies the MacWilliams-type identities) and allows to retrieve the weight enumerator in the metric considered. Additionally, we show that the partitions presented for each metric are the coarsest such partition. Hence, we refer to such a partition as the \textit{coarsest Fourier-reflexive symmetrized partition} with respect to the metric considered.\medskip

We start by introducing the set-up that we use throughout this section for all metrics. Let $\ring$ denote a finite chain ring and $\wt$ a weight function over $\ring$ extending additively to $\ring^n$. As we are looking for a Fourier-reflexive partition that enables recovering the weight enumerator of a code $\code \subseteq \ring^n$ and since the weight $\wt$ considered is additive, a natural choice is to focus on induced symmetrized partitions of $\ring^n$ (see Definition \ref{def:symmetrized-partition}).

Given a partition $\partition = I_0 \st \cdots \st I_m$ on $\ring$, we denote by $\partition_{\mathrm{sym},n}$ the induced symmetrized partition of $\ring^n$. The blocks of the symmetrized partitions are defined by the distinct decompositions $\pi$ that an $n$-tuple $x \in \ring^n$ admits. That is, we define the $\partition$-decomposition of $x \in \ring^n$ as the $(m+1)$-tuple $\pi(x) = (\pi_0(x) , \ldots , \pi_m(x))$ where, for each $i \in  \{0, \ldots, m\}$,
\begin{align}
    \pi_i(x) = \card{\set{ k \in \{ 1,\ldots , n \} \st x_k \in I_i }}.
\end{align}
Additionally, we denote by $\mathbb{D}^{\wt}_{n}$ the set of all $\partition$-decompositions. Hence, given the set $\mathbb{D}^{\wt}_{n} = \set{ \pi^{(0)}, \ldots \pi^{(D-1)}}$ for some positive integer $D$, the induced symmetrized partition $\partition_{\mathrm{sym},n}$ is given by
\begin{align}
    \partition_{\mathrm{sym},n} = P^{\wt}_{\pi^{(0)}} \st \cdots \st P^{\wt}_{\pi^{(D-1)}}.
\end{align}
That is, for every $i \in \set{0, \ldots , D-1}$,
\begin{align}
    P_{\pi^{(i)}}^{\wt}\coloneqq \set{x \in \ring^n \st \pi(x) = \pi^{(i)}} 
\end{align}
counts the number of $x\in \ring^n$ having $\partition$-decomposition $\pi^{(i)}$.
Different to the original MacWilliams identities for the Hamming metric, we aim at enumerating all codewords of the same  $\partition$-decomposition $\pi^{(i)} \in \mathbb{D}^{\wt}_{n}$. That is, for each $\pi^{(i)} \in \mathbb{D}^{\wt}_{n}$, we are interested in
\begin{align}\label{eq:enumerator_Lee}
   \mathcal{D}^{\wt}_{\pi^{(i)}}(\code)\coloneqq \card{P^{\wt}_{\pi^{(i)}}\cap \code} . 
\end{align}
Furthermore, we denote the symmetrized partition enumerator by
\begin{align}
    \mathcal{D}^{\wt}(\code) := \left(\mathcal{D}^{\wt}_{\pi^{(0)}}(\code), \ldots, \mathcal{D}^{\wt}_{\pi^{(D-1)}}(\code)\right) .
\end{align}
If the code $\code$ considered is clear from the context, we will simply write $\mathcal{D}^{\wt}$ and $\mathcal{D}^{\wt}_{\pi^{(i)}}$.

Given two $\partition$-decompositions $\pi, \rho \in \mathcal{D}^{\wt}$, with the aim of computing  the component-wise product of two tuples with decompositions $\pi = (\pi_0, \ldots, \pi_m)$ and $\rho=(\rho_0, \ldots, \rho_m)$, respectively, we denote the set of all compositions of $\rho$ with respect to $\pi$ by
\begin{align}\label{eq:compos_rho-pi}
    \mathrm{Comp}^{\wt}_{\pi}(\rho) := \Big\{t = (t_0, \ldots, t_{m})\, \Big| \, &t_i = (t_{i0}, \ldots , t_{i m}) \in \set{0, \ldots, \pi_i}^{m} \text{ with }\\  &\sum_{j = 0}^m t_{ij} = \pi_i \text{ and } \sum_{i =0}^m t_{ij} = \rho_j, 
    \text{for all } i, j \in \set{0, \ldots, m} \Big\}.
\end{align}
Figure \ref{fig:pi_rho_general} illustrates the set $\mathrm{Comp}^{\wt}_{\pi}(\rho)$, i.e., the compositions of $\rho \in \mathbb{D}^{\wt}_{n}$ with respect to $\pi \in \mathbb{D}^{\wt}_{n}$.

\begin{figure}[H]
    \centering
    \begin{tikzpicture}
        \draw (0,1) -- (11.5,1) -- (11.5,1.5) -- (0,1.5) -- cycle;
        \draw (0,0) -- (11.5,0) -- (11.5,0.5) -- (0,0.5) -- cycle;
        \draw (9, -0.2) -- (9, 0.7);
        \draw (9, 0.8) -- (9, 1.7);
        \node[left] at (0, 1.25) {$\pi = $};
        \node[left] at (0, 0.25) {$\rho = $};
        \foreach \i in {1, 2}{
            \draw (2.5*\i, -0.2) -- (2.5*\i, 0.7);
            \draw (2.5*\i, 0.8) -- (2.5*\i, 1.7);
            \draw (2.5*\i - 0.5, 0) -- (2.5*\i - 0.5, 0.5);
            \node at (2.5*\i - 1, 0.25) {$\dots$};
            \node at (11.5 - 1, 0.25) {$\dots$};
            \draw (11, 0) -- (11, 0.5);
            \foreach \j in {3, 4}{
                 \draw (2.5*\i - 0.5*\j, 0) -- (2.5*\i - 0.5*\j, 0.5);
                 \draw (8 + 0.5*\j, 0) -- (8 + 0.5*\j, 0.5); 
            }
            \foreach \j in {0, 1}{
                \node at (2.5*\i - 2.25 + 0.5*\j, 0.25) {\tiny$I_{\j}$};
                \node at (11.5 - 2.25 + 0.5*\j, 0.25) {\tiny$I_{\j}$};
                \node at (1.25 + 2.5*\j, 1.25) {\tiny$I_{\j}$};
            }
            \node at (2.5*\i - 2.25 + 2, 0.25) {\tiny$I_{m}$};
            \node at (11.5 - 0.25, 0.25) {\tiny$I_{m}$};
            \node at (1.25 + 9, 1.25) {\tiny$I_{m}$};;
            \node at (7, 1.25) {$\dots$};
            \foreach \i in {0, 1}{
                \draw [stealth-stealth] (2.5*\i, 1.6) -- (2.5*\i+2.5, 1.6);
                \node[above] at (2.5*\i+ 1.25, 1.6) {\tiny$\pi_{\i}$};
                \foreach \j in {0, 1}{
                    \draw [stealth-stealth] (2.5*\i + 0.5*\j, -0.1) -- (2.5*\i+0.5*\j+0.5, -0.1);
                    \node[below] at (2.5*\i + 0.5*\j + 0.25, -0.1) {\tiny$t_{\i \j}$};
                }
                \draw [stealth-stealth] (2.5*\i + 2, -0.1) -- (2.5*\i+2+0.5, -0.1);
                \node[below] at (2.5*\i + 2 + 0.25, -0.1) {\tiny$t_{\i m}$};
            }
            \foreach \j in {0, 1}{
                \draw [stealth-stealth] (9 + 0.5*\j, -0.1) -- (9 +0.5*\j+0.5, -0.1);
                \node[below] at (9 + 0.5*\j + 0.25, -0.1) {\tiny$t_{m \j}$};
                \draw [stealth-stealth] (9, 1.6) -- (9 +2.5, 1.6);
                \draw [stealth-stealth] (9 + 2, -0.1) -- (9 +2+0.5, -0.1);
                \node[below] at (9 + 2 + 0.25, -0.1) {\tiny$t_{mm}$};
            }
            \node[above] at (10.25, 1.6) {\tiny$\pi_{m}$};
        }
    \end{tikzpicture}
    \caption{Illustration of $\mathrm{Comp}^{\wt}_{\pi}(\rho)$.}
    \label{fig:pi_rho_general}
\end{figure}

In the following subsections we focus on the Lee, homogeneous and subfield metric. For each of them, we define an induced symmetrized partition using the above notation and show that this partition is Fourier-reflexive and allows to recover the weight enumerator. 

\subsection{Symmetrized Lee Partition for Additive Weights}
Although the MacWilliams identities for the symmetrized Lee partition enumerator are well-established \cite{astola1982leescheme, gluesing2015fourier, macwilliams1972macwilliams}, in this section we propose an alternative, yet equivalent, formulation of the MacWilliams-type identities for the Lee decomposition enumerator presented in \cite[Theorem 3.5]{gluesing2015fourier}.  Specifically, we explicitly calculate the \kr coefficients, providing a set of linear equations that enable us to recover the Lee decomposition enumerator as a function of the Lee decomposition enumerator of its dual, for any linear code over any finite chain ring  and any additive weight.\medbreak

In  this section, if the characteristic of the ring is different from 2, we label the $q^s$ elements of $\ring$ as follows:
\begin{align}
    \ring=\{ \alpha_0=0, \alpha_1,  \dots, \alpha_M,  \alpha_{M+1}, \dots , \alpha_{q^s-1}\},
\end{align}
where $M\coloneqq\floor{q^s / 2}$ and $\alpha_{q^s-i}$ represents the additive inverse of $\alpha_i$, i.e., $\alpha_{q^s-i}=-\alpha_i$. We refer to $\alpha_{q^s-i} $ as the \emph{opposite} of $\alpha_i$. To enable the recovery of the Lee weight enumerator, we introduce the Lee partition of $\ring$, in which each block consists of an element and its opposite.
\begin{definition}\label{def:weight-unit-partition}
    Given a finite chain ring $\ring$ of characteristic greater than 2,  for any $0\le i \le M$ we set  $ I^{\mathsf{L}}_{i}\coloneqq\{\alpha_i,-\alpha_i\}$ to be the subset of $\ring$ containing  $\alpha_i$ and its opposite.  We then define the \emph{Lee partition} of $\ring$ as
    \begin{align}
        \mathcal{P}^{\mathsf{L}} = I^{\mathsf{L}}_{0}\st I^{\mathsf{L}}_{1}\st \cdots \st I^{\mathsf{L}}_{M} \ .
    \end{align}
    If the characteristic of \( \ring \) is 2, then $\card{\ring}=2^r$ and every element is its own additive inverse, meaning \( \alpha_i = -\alpha_i \) for all \( i \). In this case, there is no need for relabeling, and the Lee partition 
    \begin{align}
    \mathcal{P}^\mathsf{L}= I^{\mathsf{L}}_{0}\st I^{\mathsf{L}}_{1}\st \cdots \st I^{\mathsf{L}}_{2^r} \ .
\end{align} consists of the singletons, that is $I_i^\mathsf{L}=\set{\alpha_i}$. 
\end{definition}
Following the notation introduced at the beginning of the section, we denote by $\mathcal{P}_{\mathrm{sym},n}^\mathsf{L}$ the induced symmetrized Lee partition of $\ring^n$.
\begin{proposition}\label{prop:homogeneous-reflexive}
    The symmetrized Lee partition $\mathcal P_{\mathrm{sym},n}^\mathsf{L}$ is Fourier-reflexive.
\end{proposition}
\begin{proof}
   According to \cite[Theorem 3.3]{gluesing2015fourier}, if contains $\{0\}$ as a block and \( \mathcal{P}^{\mathsf{L}} \) is Fourier-reflexive, then \( \mathcal{P}_{\mathrm{sym},n}^{\mathsf{L}} \) is Fourier-reflexive. Therefore, it suffices to prove that \( \mathcal{P}^{\mathsf{L}} \) is Fourier-reflexive.
   If the characteristic of the ring is greater than 2, the Lee partition can be interpreted as the partition induced by the orbits of the action of the group \( \{ \pm 1 \} \) on \( \ring \). It is shown in \cite[Theorem 2.6]{gluesing2015fourier} that any partition of a finite group induced by the orbits of a subgroup of its automorphism group is Fourier-reflexive.  On the other hand, if the characteristic of the ring is 2, then the Lee partition consists of singleton blocks, which is clearly a Fourier-reflexive partition \cite{wood1999duality, gluesing2015fourier}. Hence, the claim follows.
\end{proof}
 \begin{corollary}[MacWilliams-type Identities for the Symmetrized Lee Weight Enumerator]\label{thm:MacWilliams_Lee}
    Let $\ring$ be a finite chain ring, with residue field size $q=p^r$ and nilpotency index $s$. For a positive integer $n$, consider a linear code $\code \subseteq \ring^n$ and let $\dualcode$ denote its dual. For a given Lee decomposition $\rho \in \decompLq$, the number of codewords in the dual $\dualcode$ of Lee  decomposition $\rho$ is given by
    \begin{align}\label{eq:MacWilliams_Lee}
        \enumL{\rho}(\dualcode)=
        \frac{1}{\card{\code}} \sum_{\pi \in \decompLq} \krawL_{\pi}(\rho) \enumL{\pi}(\code).
    \end{align}
    Here, $\krawL_{\pi}(\rho)$ denotes the value of the \kr coefficient of $\rho$ with respect to $\pi$ and, for a given $p^s$-th root of unity $\xi$, it is given by
    \begin{align}\label{equ:kraw_Lee}
        \krawL_{\pi}(\rho) \coloneqq 
            \sum_{t \in \compL} \Bigg( \prod_{i = 0}^M \binom{\pi_i}{t_{i0}, \ldots, t_{iM}} \kappa^{\mathsf{L}} \Bigg),
    \end{align}
    where
    \begin{align}
        \kappa^{\mathsf{L}} = 
        \begin{cases}
            \prod_{j = 1}^{M-1}\left(\xi^{-\trace(\alpha_i\alpha_j)} + \xi^{\trace(\alpha_i\alpha_j)}\right)^{t_{ij}}\xi^{\trace(\alpha_i\alpha_M)}   & \text{if } p=2, \\
            \prod_{j = 1}^M \left(\xi^{-\trace(\alpha_i\alpha_j)} + \xi^{\trace(\alpha_i\alpha_j)}\right)^{t_{ij}} &\text{otherwise.}\\    
        \end{cases}
    \end{align}
    For any $k\in \set{0,\dots,M}$, $\alpha_k$  denotes any element in $I_k^{\mathsf{L}}$ and $\binom{\pi_i}{t_{i0}, \ldots, t_{iM}} $ denotes the multinomial coefficient.
\end{corollary}
\begin{proof}
     Note that, for $a \in \PL_{\pi}$, the sum $\sum_{x \in \PL_{\rho}} \chi_a(x)$ runs over all $x \in \ring^n$ with given Lee decomposition $\rho$ and is only dependent on the Lee decomposition, but not on the composition of the tuples themselves. Therefore, without loss of generality, we assume that the first $\pi_0$ entries of $a$ have  weight $0$, the next $\pi_1$ entries equal to $\pm \alpha_1$. We then find $\pi_2$ entries equals to $\pm \alpha_2$ and so on. 
    In the first $\pi_0$ positions, an $n$-tuple $x\in \ring^n$ of Lee decomposition $\rho$  can have $t_{0j} < \rho_j$ positions lying in $j\in \set{0\dots, M}$, where $\sum_{j=0}^M t_{0j} = \pi_0$. Repeating the argument iteratively for each $i \in \set{0,\dots, M}$, we deduce that the number of $x$ satisfying these conditions is then given by the product of multinomial coefficients
    \begin{align}
        \prod_{i=0}^M \binom{\pi_i}{t_{i0}, \ldots, t_{iM}}.
    \end{align}
    The formula follows from the values of the \kr coefficients  with the multiplicities provided by the composition $t \in \compH$. 
\end{proof}
In particular, if $\ring = \Zps$, i.e., $r=1$, for all $i\in \set{0,\dots,M}$, the sets $I_i^\mathsf{L}$ reduce to $I_i^\mathsf{L}=\set{\pm i}$, $\trace = \mathrm{id}$, and MacWilliams-type identities for the  Lee metric codes hold. In this case, we obtain,
\begin{align}\label{equ:kraw_Leemetric}
    \kappa^{\mathsf{L}} =
    \begin{cases}
        \prod_{j = 1}^{M-1}\left(\xi^{-ij} + \xi^{ij}\right)^{t_{ij}}\xi^{iM}   &\!\!\!\! \text{if } p=2, \\
         \prod_{j = 1}^M\left(\xi^{-ij} + \xi^{ij}\right)^{t_{ij}}&\!\!\!\!\text{else.}\\    
    \end{cases}
\end{align}
A concrete example to illustrate the calculations of the identities for Lee-metric codes $\code \subseteq \Zpsn$ is given in  \ref{app:example_Lee}.\medskip

Compared to other partitions (like, for instance, partitioning a code $\code \subseteq \ring^n$ by the codewords of the same weight), the partition into codewords of the same Lee decomposition is rather fine. 
The finer a partition is, the more complex the enumerator and the identities become.
\begin{lemma}\label{lemma:coarsest_Lee}
    Given a finite chain ring $\ring$ and a code $\code \subseteq \ring^n$.
    The symmetrized Lee partition is the coarsest symmetrized Fourier-reflexive partition that allows retrieving the Lee weight enumerator of $\code$.
\end{lemma}

\begin{proof}
    Let $\mathcal{Q}_{\text{sym},n}$ be a symmetrized, Fourier-reflexive partition of $\ring^n$ induced by $\mathcal{Q} = J_0 \st \cdots \st J_N$ of $\ring$. Assume $\mathcal{Q}_{\text{sym},n}$ is coarser than $\mathcal{P}_{\mathrm{sym},n}^\mathsf{L}$. That is, its number of blocks is strictly smaller than the number of blocks of $\mathcal P_{\mathrm{sym},n}^\mathsf{L}$. This also implies that $\mathcal{Q}$ is coarser than $\mathcal{P}^{\mathsf{L}}$, i.e., $N < M$. Therefore, there exists a block $J_i$ of $\mathcal{Q}$ whose elements do not all share the same Lee weight. Hence, by definition of the symmetrized weight, there exist $x, \Tilde{x} \in \ring^n$ with $\LW(x) \neq \LW(\Tilde{x})$ but both contained in one block of $\mathcal{Q}_{\text{sym},n}$, which makes it impossible to recover the Lee weight enumerator.
\end{proof}

Note that the Lee partition is always finer than the partition induced by any  weight $\wt$ on $\ring$. This is because, for any weight $\wt$, the weight of each element must be equal to the weight of its opposite. Consequently, the blocks of the partition of $\ring^n$ induced by $\wt$ contain the blocks of the Lee partition.
Thus, for an $n$-tuple $x \in \ring^n$ we can compute its weight $\wt(x)$ using its Lee decomposition $\pi^{\mathsf{L}}(x)$, i.e.,
\begin{align}
    \wt(x) = \sum_{i = 0}^M \pi^{\mathsf{L}}_i(x) \wt(i). 
\end{align}
Hence, the MacWilliams identities for the symmetrized Lee partition are valid for any additive weight over a finite chain ring $\ring$. However, the symmetrized Lee partition is not always the most practical partition to choose. Indeed, for other metrics, there exist coarser partitions of the ambient space that still allow for  MacWilliams-type identities and allow us to recover the weight enumerator of the code. 

\subsection{Symmetrized Homogeneous-Unit Partition}
We now focus on the homogeneous metric over $\ring$, where we assume that $\ring$ has nilpotency index $s>1$. In fact, when $s=1$, $\ring$ is a finite field, and the homogeneous metric reduces to the Hamming metric for which MacWilliams identities exist. 

Similar to the Lee metric, $n$-tuples with a given homogeneous weight can be decomposed in different ways. 
A natural choice is the \emph{symmetrized homogeneous weight partition} induced by the homogeneous weight partition on $\ring$, classifying the elements $x \in \ring^n$ of given homogeneous weight. 
That is, we partition $\ring$ into blocks containing all elements of the same homogeneous weight, i.e.,
\begin{align}
    \mathcal{Q}^{\hom} = Q_0 \st Q_1 \st Q_\lambda \; \text{ where }\; Q_i = \set{a \in \ring \st \HomW(a) = i}.
\end{align}
Hence, we can define the \textit{symmetrized homogeneous weight partition} $\mathcal{Q}^{\hom}_{\mathrm{sym},n}$ of a code $\code \subseteq \ring^n$, partitioning $\code$ into the codewords of the same homogeneous decomposition. 
However, this partition does not lead to well-defined \kr coefficients and hence the MacWilliams identities fail \cite{Heide}.
One main reason is that, given an element $a$ of homogeneous weight $1$ and a $p^s$-th root of unity $\xi$, the sum
\begin{align}
    \sum_{\substack{x \in \ring \\ \HomW(x) = 1}} \xi^{\trace(ax)}
\end{align}
is not independent on $a$.
In fact, its value depends on whether $a$ and $x$ are units or zero divisors. Therefore, we refine it by distinguishing between zero components, units, nonzero components lying in the socle $\mathcal{S}=\ideal{\gamma^{s-1}}$ of $\ring$, and components lying in $\ideal{\gamma}$ but not in $\mathcal{S}$, i.e., we define the sets
\begin{gather}
    Z := \set{0}, \; U := \ring^\times, \; S := \mathcal{S} \setminus \set{0},\; \\
    \text{and}\;  R := \ring \setminus \left(Z \cup U \cup S \right) = \langle \gamma \rangle \setminus \langle \gamma^{s-1} \rangle.
\end{gather}
We observe that the cardinalities of the sets $Z, U, S$ and $R$, respectively, are given by
\begin{align}\label{eq:cardinalities_ZUSR}
    \card{Z} = 1, \; \card{U} = q^{s-1}(q-1), \; \card{S} = q-1, \; \text{and} \; \card{R} = q^{s-1} - q .
\end{align}
This immediately partitions $\ring$ into these four sets. We call this partition of $\ring$ the \textit{homogeneous weight-unit partition} and denote it by
\begin{align}\label{eq:homogeneous-R}
    \PH := Z\st U \st S \st R .
\end{align}

Let $\PHsym = \blockH_{\pi^{(0)}} \st \ldots \st \blockH_{\pi^{(D-1)}}$ denote the induced \textit{symmetrized homogeneous weight-unit partition} on $\ring^n$, where $\pi^{(i)} \in \decompHq$, for every $i \in \set{0, \ldots, D-1}$.
Furthermore, note that every $c \in \blockH_{\pi}$ has the same homogeneous weight given by $\HomW(c) =\pi_U+\pi_R+\frac{q}{q-1}\pi_S$. Hence, the homogeneous weight enumerator $\Homwe{\code}$ can immediately be derived from the symmetrized homogeneous weight-unit enumerator.

While some Fourier-reflexive partitions over $\ring$ naturally arise from a partition of orbits of a subgroup of its automorphism group (e.g., the symmetrized Lee weight partition), this is not generally true for the homogeneous-unit partition $\PH$. Recall that the linear isometries preserving the homogeneous weight are given by
$\mathrm{Aut}^{\hom}(\ring)=U^n \rtimes S_n$.
\begin{proposition}
    The homogeneous-unit partition $\PH$ can be seen as a partition given by the orbits of a subgroup $H \leq \mathrm{Aut}^{\hom}(\ring)$ if and only if $s \in \set{1, 2, 3}$.
\end{proposition} 
\begin{proof}
   In the case where $\PH$ coincides with the Hamming weight partition (that is, if $s = 1$), the partition is clearly induced by a partition of orbits. Additionally, if $s \in \set{2, 3}$, $\PH$ is given by a partition of the orbits given by the action of $\ring^\times$ on $\ring$.

    However, if $s \geq 4$, the situation changes. Indeed, in this case, the set $R$ is given by
    \begin{align}
        R = \left( \ideal{p} \cup \cdots \cup \ideal{p^{s-2}} \right) \setminus \ideal{p^{s-1}}.
    \end{align}
    If, by contradiction, $\PH$ does coincide with a partition of an orbit, there would exist an element $r\in R$ such that $R = \set{ru \st u \in U}$. Assume that $r \in \ideal{p^i}\setminus\ideal{p^{i+1}}$, for some $i \in \{1, \ldots, s-2\}$. Then we observe
    \begin{align}
        \set{ru\st u \in U} = \ideal{p^i}\setminus \ideal{p^{i+1}} \neq R,
    \end{align}
    which contradicts the assumption. Hence, $\PH$ cannot be represented as the partition of orbits.
\end{proof}

Although $\PH$ (and hence also $\PHsym$) is not generally representable as a partition of orbits, it is still a Fourier-reflexive partition. 

\begin{proposition}\label{prop:homogeneous-reflexive}
    The symmetrized homogeneous weight-unit partition $\PHsym$ is Fourier-reflexive.
\end{proposition}
\begin{proof}
    By \cite[Theorem 3.3]{gluesing2015fourier} it suffices to show that $\PH$, as defined in \eqref{eq:homogeneous-R}, contains $\set{0}$ as a block and is Fourier-reflexive. By the definition of $\PH$ the set $\set{0}$ is a block  of $\PH$. We now show that $\PH$ is Fourier-reflexive which holds if and only if $\PH$ and $\widehat{\PH}$ consist of the same number of blocks. To show that, we identify $\ring$ and $\widehat{\ring}$ via the isomorphism $a \longmapsto \chi_a$, and, for each $a\in \ring$  we compute the sums $\sum_{x \in J} \chi_a(x)$ for each $J \in \ZUSRset$. In what follows, we will set $a\in I$ for a fixed $ I\in \set{Z,U,S,R}$,

    It is easy to see that if $J = Z$, for $I \in \ZUSRset$ and any $a \in I$, we get
    \begin{align}
        \sum_{x \in  Z} \xi^{\trace(a x)} = \xi^{0} = 1.
    \end{align}
    
    Furthermore, note that if $I = Z$ we have $a = 0$. Thus, for every $J \in \ZUSRset$ we obtain
    \begin{align}
        \sum_{x \in  J} \xi^{\trace(a x)} = \sum_{y \in  J} \xi^{0} = \card{J},
    \end{align}
    where the cardinalities $\card{J}$ with $J \in \ZUSRset$ are given in \eqref{eq:cardinalities_ZUSR}.
    Similarly, for any $(I, J) \in \set{(S, S), (S, R), (R, S)}$ we have that the product $a x = 0$ where $a \in I$ and $x \in J$. Therefore, we obtain again
    \begin{align}
        \sum_{x_{k} \in  J} \xi^{\trace(a x_{k})} = \sum_{y \in  J} \xi^{0} = \card{J}.
    \end{align}
    
    Let us now consider $I = U$ the set of units and $a \in I$. For any $J \in \set{U, S, R}$ we obtain $a J = J$. Hence, applying Corollary \ref{cor:charactersum_ideal}, we get
    \begin{align}
        \sum_{x \in J} \xi^{\trace (a x)} 
            =
        \begin{cases}
            \sum_{y \in \ideal{\gamma^{s-1}} \setminus \set{0}} \xi^{\trace (y)} = -1 & \text{if } J = S,\\
            \sum_{y \in \ring} \xi^{\trace (y)} - \sum_{y \in \ideal{\gamma}} \xi^{\trace (y)} = 0 & \text{if } J = U,\\
            \sum_{y \in \ideal{\gamma}} \xi^{\trace (y)} - \sum_{y \in \ideal{\gamma^{s-1}}} \xi^{\trace (y)} = 0 &\text{if } J = R.
        \end{cases}
    \end{align}
    
    Instead, if $J = U$,  it remains to consider the cases $I = S$ and $I = R$. In the first case, for any $a \in I$, we have $a J = I$. However, as the sum runs over $J$, we observe every element in $a J = I$ exactly $\card{U}/\card{S} = q^{s-1}$ many times and hence,
    \begin{align}
        \sum_{x \in U} \xi^{\trace (a x)} = q^{s-1} \sum_{y \in S} \xi^{\trace (y)} = q^{s-1} \cdot  (-1) = -q^{s-1}.
    \end{align}
    In the latter case $I = R$, consider any $a \in I$. Then, $a J = \ideal{\gamma^i}\setminus \ideal{\gamma^{i+1}}$. Since $\sum_{y \in \ideal{\gamma^i}\setminus \ideal{\gamma^{i+1}}} \xi^{\trace(y)} = 0$, we also observe
    \begin{align}
        \sum_{x \in J} \xi^{\trace (a x)} = 0.
    \end{align}
    
    Lastly, let us consider the case where $I = J = R$. If $a \in \langle \gamma^i \rangle \setminus \langle \gamma^{i+1} \rangle$, then $aR = (\langle \gamma^{i+1} \rangle \setminus \{0\})^w \times (\langle \gamma^{i+1}\rangle)^v$, for some positive integers $v$ and $w$, depending on whether the product $a x$ can be zero. Note that there are $|\langle \gamma^{s-i}  \rangle\setminus\langle\gamma^{s-1} \rangle| = q^i-q$ many  $x \in R$ which make the product zero. Thus, the ideal $\langle \gamma^{i+1} \rangle$ of size $q^{s-i-1}$ is seen $v=q^i-q$ many times and we are left with $|R| - (q^i-q) q^{s-i-1} = q^{s-i}-q$ many $x \in R$, which see $\langle \gamma^{i+1} \rangle \setminus \{0\}$ exactly  $w= (q^{s-i}-q)/(q^{s-i-1}-1) =q$ times. Thus, we get that
    \begin{align}
        \sum_{x \in R} \xi^{\trace{(x a)}} 
            &= 
        (q^i-q) \sum_{y \in \langle \gamma^{i+1}\rangle } \xi^{\trace{(y)}} + q \sum_{y \in \langle \gamma^{i+1} \rangle \setminus \{0\}} \xi^{\trace{(y)}}\\
            &= 
        (q^i-q) \cdot 0 + q \cdot (-1) \\
            &= -q. 
    \end{align}
    
    Table \ref{tab:krhom} gives an overview of the values of the sum $\sum_{x \in J} \chi_{a}(x)$ for all $I, J \in \ZUSRset$ and $a \in I$. The rows of the table are indexed by the blocks of $\widehat{\PH}$ and the columns by the blocks of $\PH$, showing that $\PH$ and $\widehat{\PH} $ contain the same number of blocks, which completes the proof.
    
    \begin{table}
        \caption{Values of the expression $\sum_{x \in J} \xi^{\trace(a\cdot x)}$ for any $I, J \in \ZUSRset$. }\label{tab:krhom}
            \centering
            \begin{tabular}{c|cccc}
                \backslashbox{$I$}{$J$} & $Z$ & $U$ & $S$ & $R$ \\
                \hline
                \\
                $Z$& 1 & $q^{s-1}(q-1)$ & $q-1$ & $q^{s-1}-q$ \\
                $U$& 1 & 0 & $-1$ & 0 \\
                $S$& 1 & $-q^{s-1}$ & $q-1$ &$q^{s-1}-q$  \\
                $R$& 1 & 0 & $q-1$ & $-q$ \\
            \end{tabular}
    \end{table}
\end{proof}

By \cite[Theorem 3.5]{gluesing2015fourier} the MacWilliams-type identities exist for the symmetrized homogeneous weight-unit enumerator and can be summarized as follows.

\begin{corollary}[MacWilliams-type Identities for the Symmetrized Homogeneous Weight-unit Enumerator]\label{thm:MacWilliams_Hom}
    Let $\ring$ be a finite chain ring with residue field size $q=p^r$ and nilpotency index $s$.
    Consider a linear code $\code \subseteq \ring^n$ and let $\dualcode$ denote its dual. For a given homogeneous weight-unit decomposition $\rho \in \decompHq$, the number of codewords in the dual $\dualcode$ of homogeneous weight-unit decomposition $\rho$ is given by
    \begin{align}\label{eq:MacWilliams_Hom}
        \enumH{\rho}(\dualcode) = \frac{1}{\card{\code}} \sum_{\pi \in \decompHq} \krawH_{\pi}(\rho) \enumH{\pi}(\code).
    \end{align}
    Here, $\krawH_{\pi}(\rho)$ denotes the value of the \kr coefficient of $\rho$ with respect to $\pi$, and for a given $p^s$-th root of unity $\xi$  is given by
    \begin{align}
        \krawH_{\pi}(\rho)=\sum_{t \in \compH} \Bigg(  & \prod_{i \in \ZUSRset } \binom{\pi_i}{t_{iZ}, \ldots, t_{iR}} \Bigg) \kappa^{\mathsf{hom}},
    \end{align}
    where, given that $t_{UU}=t_{UR} = t_{RU} = 0$, we have
    \begin{align}
        \kappa^{\mathsf{hom}} = &(-1)^{t_{US}} (-q^{s-1})^{t_{SU}} \left(q^{s-1}(q-1)\right)^{t_{ZU}}
        \left(q^{s-1} -q\right)^{t_{ZR}} (q-1)^{t_{ZS}+t_{SS}+t_{SR}+t_{RS}} (-q)^{t_{RR}}.
    \end{align}
\end{corollary}
The proof follows a similar arguments as the proof of Corollary \ref{thm:MacWilliams_Lee}.

In  \ref{app:example_Hom}, we give an example to get an idea of the formula stated above.
Finally, note that the homogeneous unit-weight partition defines the coarsest partition that allows us to obtain MacWilliams-type identities for the homogeneous metric.
\begin{lemma}\label{lemma:coarsest_Hom}
    Given a finite chain ring $\ring$ and a code $\code \subseteq \ring^n$.
    The symmetrized homogeneous weight-unit partition $\PHsym$ is the coarsest symmetrized partition which is Fourier-reflexive and allows to recover the homogeneous weight enumerator $\code$.
\end{lemma}
\begin{proof}
    Let us assume there exists partition $\mathcal{Q}_{\text{sym},n}$ of $\ring^n$ which is coarser than $\PHsym$. Hence, the underlying partition $\mathcal{Q}$ of $\ring$ consists of $3$ or fewer blocks. If $\mathcal{Q}$ is not the homogeneous weight partition, by similar arguments as in the proof of Lemma \ref{lemma:coarsest_Lee}, the homogeneous weight enumerator cannot be recovered. On the other hand, if $\mathcal{Q}$ is the homogeneous weight partition, then $\mathcal{Q}_{\text{sym},n}$ is not Fourier-reflexive in general (see Theorem \cite{Heide}).
\end{proof}

\subsection{Symmetrized Subfield-Trace Partition}
Let us focus on the subfield metric defined over the finite field $\field_{p^r}$. Since $\field_{p^r}$ is a finite chain ring with $s=1$ and the subfield weight is additive, the {symmetrized Lee partition} allows us to obtain MacWilliams-type identities for the subfield weight. However, this is not the coarsest partition we can choose. Again, a natural choice is to consider the subfield decomposition of the codewords defined as follows:
given a code $\code \subseteq \field_{p^r}^n$ the \textit{subfield decomposition} of a codeword $c \in \code$ is a triple $(s_0(c), s_1(c), s_{\lambda}(c))$, where for each $i \in \set{0, 1, \lambda}$ we define
\begin{align}
    s_i(c) := \card{\set{j \in \{ 1, \ldots, n\} \st \SubW(c_j) = i}}.
\end{align}
We can then partition the code into blocks containing the codewords with the same subfield weight decomposition. The sizes of the blocks determine the  \textit{ symmetrized subfield decomposition enumerator} of a code $\code$ and we denote it by $\SDe{\code}$. 

\begin{example}\label{ex:counterexample_subfield}
    Consider the finite field $\field_8 = \field_2(\alpha)$ where $ \alpha^3 = \alpha + 1$ and the two codes $\code_1 = \ideal{\begin{pmatrix} 1, \alpha, \alpha^2 \end{pmatrix}}$ and $\code_2 = \ideal{\begin{pmatrix} \alpha^2 , 1, \alpha^2+1 \end{pmatrix}}$. We denote by $\dualcode_1$ and $\dualcode_2$ the dual of $\code_1$ and $\code_2$, respectively. Table \ref{tab:counterexample_subfield} then shows that  $\SDe{\code_1} = \SDe{\code_2}$, but $\SDe{\dualcode_1} \neq \SDe{\dualcode_2}$.
    \begin{table}
    \caption{Subfield weight enumerator for the codes $\code_1, \code_2$ and their respective duals $\dualcode_1$ and $\dualcode_2$.}\label{tab:counterexample_subfield}
        \centering
        \begin{tabular}{c|c|c|c|c}
            subfield weight decomp. & $\SDe{\code_1}$ & $\SDe{\code_2}$ & $\SDe{\dualcode_1}$ & $\SDe{\dualcode_2}$  \\
            \hline
            (3, 0, 0) & 1 & 1 & 1 & 1 \\
            (0, 3, 0) & 0 & 0 & 0 & 1 \\
            (0, 0, 3) & 4 & 4 & 27 & 26\\
            (2, 1, 0) & 0 & 0 & 0 & 0 \\
            (2, 0, 1) & 0 & 0 & 0 & 0 \\
            (1, 2, 0) & 0 & 0 & 0 & 0 \\
            (1, 0, 2) & 0 & 0 & 15 & 15 \\
            (0, 2, 1) & 0 & 0 & 3 & 0 \\
            (0, 1, 2) & 3 & 3 & 12 & 15 \\
            (1, 1, 1) & 0 & 0 & 6 & 6 
        \end{tabular}
    \end{table} 
\end{example}
Example \ref{ex:counterexample_subfield} also shows that partitioning a code $\code$ into codewords of the same subfield weight does not yield to MacWilliams identities. Indeed, for $\lambda > 1$, Example \ref{ex:counterexample_subfield} shows that the subfield weight enumerators of $\code_1$ and $\code_2$ coincide (i.e.,  $\Swe{\code_1} = \Swe{\code_2}$), whereas the subfield weight enumerators of their duals are different (i.e, $\Swe{\dualcode_1} \neq \Swe{\dualcode_2}$).

The symmetrized subfield decomposition defined above is the coarsest partition that allows to recover the subfield weight of the elements of a code. As seen in Example \ref{ex:counterexample_subfield}, it is not Fourier-reflexive. However, we can define a Fourier-reflexive partition consisting of only four blocks. More explicitly, with an eye on Example \ref{ex:char}, we will distinguish the elements in $\Fprn$ not only by their subfield weight, but also by their trace.
As we still want to be able to recover the subfield weight of a vector, we sub-partition the elements of subfield weight $\lambda$ (i.e., the elements in $\field_{p^r}\setminus \Fp$) into two sets: the first one contains all such elements with trace equal to $0$ and the second one contains all remaining elements. That is, we partition  $\field_{p^r}$ as $\PS= \IS_0 \st \IS_1 \st \IS_2 \st \IS_3$,  where
\begin{gather}\label{eq:subfield-trace_Fpr}
    \IS_0 := \set{0}, \quad \IS_1 := \Fp^{\times}, \quad \IS_2 := \set{x \in \Fpr \setminus \Fp \st \trace(x) = 0}\quad \text{and} \\ \IS_3 := \set{x \in \Fpr \setminus \Fp \st \trace(x) \neq 0}.
\end{gather}
We call $\PS$ the \textit{subfield-trace} partition of $\field_{p^r}$. Recall 
that for each $x \in \Fp$ we have $\trace(x) = rx \mod p$ and that the trace function is uniformly distributed among $\Fp$. Hence, the following is a simple consequence.
\begin{corollary}\label{cor:cardinalities_subfield}
    Let $\Fprn$ be a finite field and $p$ a prime number. If $r \centernot\mid p$, we have that $| \IS_2 | = p^{r-1} - 1$ and $| \IS_3 | = (p^{r-1} - 1)(p-1)$. On the other hand, if $r \st p$, it holds that $| \IS_2 | = p^{r-1} - p$ and $| \IS_3 | = p^{r-1}(p-1)$.
\end{corollary}

Let $\decompSq = \set{\pi^{(0)}, \ldots, \pi^{(D-1)}}$ denote the set of $\PS$-decompositions. We denote the \textit{induced symmetrized subfield-trace partition} $\PSsym$ of $\Fprn$ by
\begin{align}\label{eq:subfield-trace_partition}
   \PSsym := \blockS_{\pi^{(0)}} \st \cdots \st \blockS_{\pi^{(D-1)}}.
\end{align}
Note also here that every codeword $c \in \code$ with $c \in \blockS_{\pi}$ has the same subfield weight which is calculated by $\wt_{\lambda}(c) = \pi_1 + (\pi_2 + \pi_3)\lambda$. Thus, the subfield-weight enumerator $\SDe{\code}$ of a code $\code \subseteq \Fprn$ can immediately be derived from its subfield-trace enumerator.

\begin{proposition}\label{prop:reflexive-subfield}
    The symmetrized subfield-trace partition $\PSsym$ is Fourier-reflexive.
\end{proposition}
\begin{proof}
    Similar to the proof of Proposition \ref{prop:homogeneous-reflexive} we show that the subfield-trace partition $\PS$ is Fourier-reflexive by calculating the sums $\sum_{x \in \IS_j} \chi_a(x)$, for each $a \in \ring$ and each $j \in \set{0, \ldots, 3}$. 
    
    For $a = 0$ and for every $j \in \set{0, \ldots, 3}$ we observe that
    \begin{align}
        \sum_{x \in \IS_j} \chi_a(x) = |\IS_j|,
    \end{align}
    where the values of $\card{\IS_j}$, for $j = 2, 3$, can be found in Corollary \ref{cor:cardinalities_subfield}.
    Similarly, for any $i \in \set{0, \ldots, 3}$ and $a \in \IS_i$ we have 
    \begin{align}
        \sum_{x \in \IS_0} \chi_a(x) = 1.
    \end{align}
    For the remaining cases, we make use of Lemmas \ref{lemma:schur_code} and \ref{cor:charactersum_ideal} and we identify whether the character $\chi_a(x)$ is  generating for the subgroup $ \field_p\subset\field_{p^r}$. Recall that, a character $\chi_a$ of $\field_{p^r}$ is generating for the subgroup $H $ if $H$ is not contained in $\ker( \chi_a)$.
    Let us  first consider $a \in \IS_1$, that is $a\in  \field_p^\times$. By $\field_{p}$-linearity of the trace function, we have $\trace(ax) = a\trace(x)$ for any $x \in \field_{p^r}$. If $r \st p$, for any $x \in \field_{p}$ we have $\trace(x) = 0$, i.e., $\chi_a$ is non-generating for $\field_{p}$. Hence, by Lemma \ref{lemma:schur_code}, $\sum_{x \in \field_{p}} \chi_a(x) = p$. Therefore,
    \begin{align}
        \sum_{x \in \field_{p}^{\times}} \chi_a(x) = p - 1.
    \end{align}
    On the other hand, if $r \centernot\mid p$, then $\chi_a$ is generating for $\field_{p}$ and 
    \begin{align}
        \sum_{x \in \field_{p}^{\times}}\chi_a(x) = -1.
    \end{align}
    Furthermore, it is easy to see that for $r \st p$ we have that  $\IS_2 \cup \field_{p}$ (respectively $\IS_2 \cup \set{0}$ if $r \centernot\mid p$) is a subgroup of $\field_{p^r}$ and for each $x \in \IS_2$ it holds that $\trace(x) = 0$. Thus, given $a\in \field_p$, from Corollary \ref{cor:cardinalities_subfield} follows
    \begin{align}
        \sum_{x \in \IS_2}\chi_a(x) = 
        \begin{cases}
            \sum_{x \in \IS_2 \cup \field_{p}} \chi_a(x) - \sum_{x \in \field_{p}} \chi_a(x)= |\IS_2| = p^{r-1}-p & \text{if } r \st p,\\
            \sum_{x \in \IS_2 \cup \set{0}} \chi_a(x) - \chi_a(0) = |\IS_2| = p^{r-1}-1 & \text{otherwise}.
        \end{cases}
    \end{align}
    Now, let   $a \in \IS_2$, i.e., we have $\trace(a) = 0$. From the  $\field_{p}$-linearity of the trace,  we have that  $\trace(ax) = \trace(a)x = 0$ for any $a \in \IS_2, x\in \field_p^\times$ and hence, $\chi_a$ is  non-generating for $\field_{p}$ for any $a$ . Then, by Lemma \ref{lemma:schur_code}, we get that
    \begin{align}
        \sum_{x \in \field_{p}^{\times}} \chi_a(x) = \sum_{x \in \field_{p}} \chi_a(x) - \chi_a(0) = p-1.
    \end{align}
    We now compute $\sum_{x\in \IS_2} \chi_a(x)$. Recall that $\IS_2 \cup \field_{p}$ (respectively $\IS_2 \cup \set{0}$) is a subgroup of $\field_{p^r}$ if $r\st p$ (respectively $r \centernot\mid p$). Since  $\IS_2$ is not closed under multiplication, $\chi_a$ is generating over $\IS_2\cup\field_{p}$ for $r \st p$ (respectively over $\IS_2\cup\set{0}$ for $r \centernot\mid p$) and thus,
    \begin{align}
        \sum_{x \in \IS_2}\chi_a(x) = 
        \begin{cases}
            \sum_{x \in \IS_2 \cup \field_{p}} \chi_a(x) - \sum_{x \in \field_{p}} \chi_a(x) =  -p & \text{if } r \st p,\\
            \sum_{x \in \IS_2 \cup \set{0}} \chi_a(x) - \chi_a(0)  = -1& \text{otherwise}.
        \end{cases}
    \end{align}
    Finally, we assume that $a \in \IS_3$. Note that, if $x=1 \in \field_{p}$, then  $\trace(ax) = \trace(a) \neq 0$. Hence, $\chi_a$ is generating  for $\field_{p}$ over $\field_{p^r}$ which implies (using Corollary \ref{cor:charactersum_ideal}) that
    \begin{align}
        \sum_{x \in \field_{p}^{\times}} \chi_a(x) = -1.
    \end{align}
   Similarly to the previous cases, we   obtain
    \begin{align}
        \sum_{x \in \IS_2}\chi_a(x) = 
        \begin{cases}
            \sum_{x \in \IS_2 \cup \field_{p}} \chi_a(x) - \sum_{x \in \field_{p}} \chi_a(x) = 0 & \text{if } r \st p,\\
            \sum_{x \in \IS_2 \cup \set{0}} \chi_a(x) - \chi_a(0) =  -1 & \text{otherwise}.     
        \end{cases}
    \end{align}
    Observing that $\sum_{x \in \IS_3} \chi_a(x) = \sum_{x \in \field_{p^r}} \chi_a(x) - \sum_{x \in \field_{p}} \chi_a(x) - \sum_{x \in \IS_2} \chi_a(x)$, yields the last cases for each $a$.
    
    The values of $\sum_{x \in \IS_j} \chi_a(x)$ are summarized in Table \ref{tab:kraw_subfield_r-divides-p} for $r\st p$, respectively in Table \ref{tab:kraw_subfield_r-NOTdivides-p} whenever $r \centernot\mid p$. In either cases we again interpret the rows of the table as the blocks of $\widehat{\PS}$ and columns as the blocks of $\PS$. Hence, the number of blocks of $\widehat{\PS}$ equals the number of blocks of $\PS$. Thus, $\PS$ is Fourier-reflexive. Additionally, the set $\set{0}$ is, by definition, a block of $\PS$. Hence, by \cite[Theorem 3.3]{gluesing2015fourier}, $\PSsym$ is Fourier-reflexive too.
    \begin{table}[H]
        \centering
        \begin{tabular}{c|cccc}
                & $0$ & $1$ & $2$ & $3$ \\
            \hline
            $0$ & $1$ & $1$ & $1$ & $1$ \\
            $1$ & $p-1$ & $p-1$ & $p-1$ & $-1$ \\
            $2$ & $p^{r-1}-p$ & $p^{r-1}-p$ & $-p$ & $0$ \\
            $3$ & $p^{r-1}(p-1)$ & $-p^{r-1}$ & $0$ & $0$ \\
        \end{tabular}
        \caption{Values of $\sum_{x \in \IS_j} \chi_a$ for $a \in \IS_i$, $i, j \in \set{0, \ldots, 3}$ and $r \st p$.}
        \label{tab:kraw_subfield_r-divides-p}
    \end{table}
    \begin{table}[H]
        \centering
        \begin{tabular}{c|cccc}
                & $0$ & $1$ & $2$ & $3$ \\
            \hline
            $0$ & $1$ & $1$ & $1$ & $1$ \\
            $1$ & $p-1$ & $-1$ & $p-1$ & $-1$ \\
            $2$ & $p^{r-1}-1$ & $p^{r-1}-1$ & $-1$ & $-1$ \\
            $3$ & $(p^{r-1}-1)(p-1)$ & $-(p^{r-1}-1)$ & $1-p$ & $1$ \\
        \end{tabular}
        \caption{Values of $\sum_{x \in \IS_j} \chi_a$ for $a \in \IS_i$, $i, j \in \set{0, \ldots, 3}$ and $r \centernot\mid p$.}
        \label{tab:kraw_subfield_r-NOTdivides-p}
    \end{table}
\end{proof}
For some cases, this partition is induced by partitions of an orbit. However, as we could not verify this for all instances, we leave it as an open question.

We are now able to state the MacWilliams identities for the subfield-trace enumerator whose existence is given by \cite[Theorem 3.5]{gluesing2015fourier}.
\begin{corollary}[MacWilliams-type Identities for the Symmetrized Subfield-Trace Enumerator]\label{thm:MacWilliams_subfield}
    Consider a linear code $\code \subseteq \Fprn$ and let $\dualcode$ denote its dual. For a given $\rho \in \decompSq$, the number of codewords in the dual $\dualcode$ of subfield-trace decomposition $\rho$ is given by
    \begin{align}
        \enumS{\rho}(\dualcode) = \frac{1}{\card{\code}} \sum_{\pi \in \decompSq}\krawS_{\pi}(\rho) \enumS{\pi}(\code).
    \end{align}
    Here $\krawS_{\pi}(\rho)$ denotes the value of the Krawtchouk coefficient of $\rho$ with respect to $\pi$ which, for a given $p$-th root of unity $\xi$, is defined as
    \begin{align}
        \krawS_{\pi}(\rho) = \sum_{t \in \compS}\left( \prod_{i = 0}^3 \binom{\pi_i}{t_{i0}, \ldots, t_{i3}} \right) \cdot \kappa^{\lambda},
    \end{align}
    where $\kappa^{\lambda}$ takes the values
    \begin{gather}
            (p-1)^{t_{01} + t_{11} + t_{21} + t_{03}} (-1)^{t_{31}} (-p)^{t_{22}} (p^{r-1}-p)^{t_{02} + t_{12}} (-p^{r-1})^{t_{13}} (p^{r-1})^{t_{03}}, \text{ and } \\
            (p-1)^{t_{01} + t_{03}} (-1)^{t_{11} + t_{22} + t_{33} + t_{32}} (1-p)^{t_{23}} (p^{r-1}-1)^{t_{02} + t_{03} +t_{12}+t_{21}} (1- p^{r-1})^{t_{13}},
    \end{gather}
    if $r \st p$ and $r \not \; \mid p$, respectively.
\end{corollary} 
\begin{proof}
    The multinomial appearing in the expression of the \kr coefficient $\krawS_{\pi}(\rho)$ is derived as in the proof of Corollary \ref{thm:MacWilliams_Hom}. The formula follows using the values of Tables \ref{tab:kraw_subfield_r-divides-p} or \ref{tab:kraw_subfield_r-NOTdivides-p}, respectively, with the multiplicities provided by the composition $t \in \compS$. 
\end{proof}
An illustrative example can be found in  \ref{app:example_Sub}.\medskip

Lastly, we want to highlight that the symmetrized subfield-trace partition (see \eqref{eq:subfield-trace_partition}) is indeed the coarsest partition that allows for the MacWilliams identities to exist and to recover the subfield weight enumerator. The proof follows similar arguments to the proof of Lemma \ref{lemma:coarsest_Hom} and is therefore omitted.
\begin{lemma}\label{lemma:coarsest_subfield}
    Given a finite field $\Fpr$ and a code $\code \subseteq \Fpr^n$.
    The symmetrized subfield-trace partition is the coarsest symmetrized partition which is Fourier-reflexive and allows to recover the subfield-weight enumerator of $\code$.
\end{lemma}

\section{Linear Programming Bound}\label{sec:lpbound}
One of the tightest bounds for the size of linear codes is given by the Linear Programming (LP) bound. Such bounds have been introduced by Delsarte \cite{delsarte1998association} and usually stem from association schemes. However, given the MacWilliams identities, one can easily derive an LP bound. For the homogeneous and the subfield metric no such LP bound has been stated. An exception is the Lee metric, where an association scheme using the Lee decomposition has been employed to derive a Lee-metric LP bound \cite{astola1982leescheme,patrick1986lee}. For the homogeneous metric over a Frobenius ring a different approach to derive an LP bound has been presented in \cite{byrne2007linear}. 
In this Section, we derive LP bounds stemming from the MacWilliams identities presented in Section \ref{sec:new_macwilliam}.

Given a finite chain ring $\ring$, an additive weight $\weight: \ring^n \to \QQ$ and an $\ring$-linear code $\code$, recall the definition of the \emph{decomposition enumerator} of $\code$ 
\begin{align}
    \enumW{\pi}(\code) = \card{\set{c \in \code \st \pi^{\wt}(c) = \pi}}
\end{align}
with respect to the weight $\wt$. Moreover, if $\card{\decompWq}=D$, we denote by $\wtdec{0}=(n,0,\dots,0)$ the decomposition of the all-zero vector, and all the other decompositions in $\decompWq$ by $\wtdec{1},\dots,\wtdec{D-1}$. 

The MacWilliams-type identities for the Lee decomposition enumerator, given in Corollary \ref{thm:MacWilliams_Lee}, allow us to straightforwardly derive an LP bound for any $\ring$-linear code and any additive weight. Moreover, the MacWilliams-type identities given in \Cref{thm:MacWilliams_Hom} and \Cref{thm:MacWilliams_subfield} allow to derive an LP bound for the homogeneous and subfield metric respectively,  which is less complex (i.e., that involves fewer variables).

From now on,  $ B_{\ring}(n,\dist_{\wt}) $ denotes the maximum number of codewords of an $\ring$-linear code of  length $ n$ and minimum $\weight$-distance $\dist_{\wt}$.
\begin{theorem}[Linear Programming Bound]\label{thm:lp_bound}
    For a linear code $\code \subseteq \ring^n$ of minimum distance $\dist_{\wt}$ it holds that
    \begin{align}
        B_{\ring}(n,\dist_{\wt})\le \max \left( \sum_{i=0}^{D-1} \enumW{\wtdec{i}}(\code)\right),
    \end{align}
    where the maximum is taken over all $\enumW{\wtdec{i}}(\code)$ subject to the following conditions: 
    \begin{itemize}
        \item $\enumW{\wtdec{0}}(\code)=1$;
        \item $\enumW{\wtdec{i}}(\code)\ge 0$ for all $1\le i\le D-1$ and $\enumW{\wtdec{i}}(\code)=0$ for all $i \in I $;
        \item 
           $ \sum_{i=1}^{D-1} \enumW{\wtdec i}(\code) K_{\wtdec j}(\wtdec i)\ge 0$,
    \end{itemize}
    where $I=\{i \ \st \  x \in \ring^n \text{ such that }\pi(x)=\wtdec i \text{ and } \weight(x)<\dist_{\wt}\}$.
\end{theorem}
If $\wt$ is an additive weight of $\ring^n$ and $\enumW{\pi}(\code)$ is the Lee partition enumerator, the LP bound  requires a large number of variables. Thus, the program cannot run efficiently and is not computationally feasible.

\subsection{Lee metric}
The Lee scheme, introduced in \cite{astola1982leescheme} and  investigated in \cite{astola2015sharpening,astola2016leeLPbound,patrick1986lee}, provides a framework for studying the maximum size of a linear Lee-metric code. In the following we will refine the Lee-metric  LP bound \Cref{thm:lp_bound}. 

Throughout this section, let $\ring=\Zpsn$ and let $M$ be the maximum value of the Lee-weight in $\Zps$. 
Given a $\Zps$-linear code $\code$ of length $n$, let $c\in \code$ denote a codeword. In the Hamming metric, whenever we multiply each component of $c$ by a unit $r \in \units{p^s}$, we obtain a codeword $r\cdot c$ of the same Hamming weight as $c$. This property does not hold in the Lee metric. However, the Lee decompositions of $c$ and $r\cdot c$ are related.
Indeed, the Lee decomposition of $r \cdot c$ is a permutation of the Lee decomposition of $c$. That is, if $\piL(c) = (\piL_0, \dots, \piL_{M})$ is the Lee decomposition of $c$, then the Lee decomposition of $r \cdot c$ is given by 
\begin{align}
    \piL(r \cdot c) = (\piL_0, \piL_{{\sigma_r}(1)}, \dots, \piL_{\sigma_r(M-1)}) =: \sigma(\piL(r\cdot c)),
\end{align}
where $\sigma_r(i) = k$ such that $rk \equiv i \pmod{p^s}$.

\begin{lemma}[\text{\cite[Lemma 1]{astola2014linear}}]
    Given a Lee decomposition $\pi \in \decompLq$, a unit element $r\in \units{p^s}$  and $\sigma_r$ a permutation on the set $\set{1, \ldots, M}$, we define $\sigma_r(\pi):=(\pi_0, \pi_{{\sigma_r}(1)}, \dots, \pi_{\sigma_r(M-1)})$. Then, for any linear code $\code \subseteq \Zpsn$, it holds that $\enumL{\pi}(\code)=\enumL{\sigma(\pi)}(\code)$.
\end{lemma}

For any Lee decomposition $\pi \in \decompLq$, we define the set
\begin{align}
    S(\pi) := \set{\sigma_r(\pi) \st r \in \units{p^s},\ \sigma_r(i)=k \text{ with } kr\equiv i \pmod{p^s}}.
\end{align}
Since $\enumL{\pi}(\code) = \enumL{\rho}(\code)$ for every $\rho \in S(\pi)$, an additional constraint arises in the formulation of the LP problem for linear Lee-metric codes.

Let $ B_{p,1,s,n}(n,\LD)$ denote the maximum number of codewords in a linear code of length $n$ and minimum Lee-distance $\LD$ over $\Zps$.
\begin{theorem}[Linear Programming Bound for Lee-metric Codes] 
    For a linear code $\code \subseteq \Zpsn$ of minimum Lee distance $\LD$ it holds that
    \begin{align}
         B_{\Zps}(n,\LD)\le \max \left(\sum_{i=0}^{D-1} \enumL{\wtdec{i}}(\code)\right),
    \end{align}
    where the maximum is taken over all $\enumL{\wtdec{i}}(\code)$ subject to the following conditions: 
    \begin{itemize}
        \item $\enumL{\wtdec{0}}(\code)=1$;
        \item $\enumL{\wtdec i}(\code)\ge 0$ for all $1\le i\le D-1$ and $\enumL{\wtdec{i}}(\code)=0$ for all $i \in I $;
        \item $\enumL{\wtdec i}(\code)=\enumL{\wtdec j}(\code)$ for all $\wtdec j \in S(\wtdec i)$;
        \item and \begin{align}
            \sum_{i=1}^{D-1} \enumL{\wtdec i}(\code) \krawL_{\wtdec j}(\wtdec i)\ge
            \begin{cases}
                 -\binom{n}{\wtdec i_1, \ldots, \wtdec i_M}2^{n-\wtdec i_0-\wtdec i_M }& \text{if } p=2, 
                \\ -\binom{n}{\wtdec i_1, \ldots, \wtdec i_M}2^{n-\wtdec 0_i } & \text{otherwise,}
            \end{cases}
        \end{align}
    \end{itemize}
    where $I=\{i \ \st \  x \in \Zpsn \text{ such that }\pi(x)=\wtdec i \text{ and } \LW(x)<\LD\}$.
\end{theorem}
The last constraint is derived from the fact that 
\begin{align}
    \enumL{\wtdec j}(\dualcode) = \sum_{i=0}^{D-1} \enumL{\wtdec i}(\code)\krawL_{\wtdec j} (\wtdec i) \ge 0
\end{align}
and from computing the \kr coefficient of  $\wtdec 0$ with respect to $\wtdec{j}$.\\ 
However, note that this formulation of the LP problem still does not result in an efficient program due to the high number of variables involved.
\section{Conclusion}\label{sec:conclusion}
In this paper, we investigated MacWilliams-type identities for codes over finite chain rings equipped with additive metrics. Our approach is based on studying partitions of the ambient space that refine the corresponding weight partitions while remaining Fourier-reflexive. For the Lee metric, we showed that the partition induced by the action of the full group of linear isometries is already the coarsest Fourier-reflexive refinement of the weight partition. In contrast, for the homogeneous and subfield metric, we identified new symmetrized partitions that are significantly coarser than the orbit partitions induced by subgroups of the linear isometry group, yet remain Fourier-reflexive. These partitions are, moreover, the coarsest symmetrized partitions for which MacWilliams-type identities hold. The presented identities can be used to formulate LP bounds in each of the metrics.

In this work we laid the main focus on the Lee, homogeneous and subfield metric. Even though the MacWilliams-type identities for the Lee partition enumerator allows to derive the weight enumerator of any additive weight over any finite chain ring, for specific weights there are more natural choices on how to partition the code. Future work, therefore, includes the study of other weights, such as the Bachoc weight \cite{bachoc}, Krotov weight \cite{krotov1,krotov2} or the overweight \cite{gassner}.

\appendix
\section{Examples}\label{app:examples}
In this section we provide computational examples of the partition-based Mac-Williams identities in the Lee, homogeneous and subfield metric presented in Corollaries \ref{thm:MacWilliams_Lee}, \ref{thm:MacWilliams_Hom} and \ref{thm:MacWilliams_subfield}, respectively.

\subsection{Lee Weight Partition}\label{app:example_Lee}
We start by giving an example of the MacWilliams-type identities for the Lee weight partition given in Corollary \ref{thm:MacWilliams_Lee}.

\begin{example}
    Consider the code $\code \subseteq (\Zmod{9})^3$ generated by $G = \begin{pmatrix} 3 & 2 & 8 \end{pmatrix}$. That is,
    \begin{align}
        \code  = \big\{(0, 0, 0), &(0, 3, 3), (0, 6, 6), (3, 2, 8), (6, 7, 1), 
        (6, 4, 7),(3, 5, 2),  (3, 8, 5), (6, 1, 4)\big\}.
    \end{align}
    Note that there are five distinct Lee decompositions in $\code$, namely, 
    \begin{gather}
        \pi^{(0)} = (3, 0, 0, 0, 0), \ \pi^{(1)} = (1, 0, 0, 2, 0), \ \pi^{(2)} = (0, 1, 1, 1, 0), \\ \pi^{(3)} = (0, 0, 1,1,1) \ \text{ and } \ \pi^{(4)}=(0, 1, 0,1,1) .
    \end{gather}
    Additionally, we have that $\enumL{\pi^{(0)}}(\code) = 1$ and $\enumL{\pi^{(j)}}(\code) = 2$ for $j = 1,2, 3, 4$.
    The dual $\dualcode$ of $\code$ is generated by the matrix
    \begin{align}
        G^{\perp} = \begin{pmatrix} 1 & 0 & 3 \\ 0 & 1 & 2 \end{pmatrix}.
    \end{align}
    Let $\rho = (2, 0, 0, 1, 0)$ be a Lee decomposition in $(\Zmod{9})^3$. By calculating the elements of $\dualcode$, we observe two codewords in $\dualcode$ of Lee decomposition $\rho$, namely $(3, 0, 0)$ and $(6, 0, 0)$. Hence, $\enumL{\rho}(\dualcode) = 2$.
    
    We now compute the \kr coefficient $\krawL_{\pi}(\rho)$, as stated in \eqref{equ:kraw_Leemetric}, for $\pi = \pi^{(j)}$ with $j = 0, \ldots , 4$. Let us start with $\pi = \pi^{(0)} = (3, 0, 0, 0, 0)$. Notice, since $\pi_i = 0$ for every $i = 1, \ldots, 4$, the only possibility for $t = (t_{0}, \ldots t_{M})$ is given by $t_0 = \rho$ and $t_i = (0, \ldots, 0)$.
    That is, $t_{ij} = 0$ and thus $(\xi^{-ij} + \xi^{ij})^{t_{ij}} = 1$ for every $i > 0$ and every $j$, as well as $\binom{\pi_i}{t_{i0}, \ldots , t_{iM}} = 1$. Hence, for $\pi = \pi^{(1)}$ we have 
    \begin{align}\label{ex:Lee_pi1}
        \prod_{i = 0}^M \binom{\pi_i}{t_{i0}, \ldots, t_{iM}}  \prod_{j = 1}^M\left(\xi^{-ij} + \xi^{ij}\right)^{t_{ij}}
            &=
        \binom{\pi_0}{t_{00}, \ldots , t_{0M}} \left(1 + 1 \right)^{t_{03}} \\
            &= 
        \binom{3}{2, 1}(1+1) \\
            &= 6.
    \end{align}
    Consider now the partition $\pi = \pi^{(1)} = (1, 0, 0, 2, 0)$. There are two compositions $t = (t_{0}, \ldots, t_{M}), s = (s_{0}, \ldots, s_{M}) \in \compL$, given by
    \begin{align}
        &t_0 = (1, 0, 0, 0, 0), \; t_{3} = (1, 0, 0, 1, 0)  \text{ and } t_1 = t_2 = t_4 = (0, \ldots, 0),\\ 
        &s_0 = (0, 0, 0, 1, 0),\; s_{3} = (2, 0, 0, 0, 0) \text{ and } s_1 = s_2 = s_4 = (0, \ldots, 0).
    \end{align}
    Hence, we get
    \begin{align}\label{ex:Lee_pi2}
        \sum_{t \in \compL} \left(\prod_{i = 0}^4 \binom{\pi_i}{t_{i0}, \ldots, t_{i4}}  \prod_{j = 1}^4\left(\xi^{-ij} + \xi^{ij}\right)^{t_{ij}}\right)&=
        \left( \binom{\pi_0}{t_{00}, \ldots, t_{04}} \prod_{j = 1}^4\left(\xi^{0} + \xi^{0}\right)^{t_{0j}} \binom{\pi_3}{t_{30}, \ldots, t_{34}}\prod_{j = 1}^M\left(\xi^{-3j} + \xi^{3j}\right)^{t_{3j}} \right.\\
            &\quad\qquad 
        + \left. \binom{\pi_0}{s_{00}, \ldots, s_{04}} \prod_{j = 1}^4\left(\xi^{0} + \xi^{0}\right)^{s_{0j}} \binom{\pi_3}{s_{30}, \ldots, s_{34}}\prod_{j = 1}^M\left(\xi^{-3j} + \xi^{3j}\right)^{s_{3j}}  \right)\\
            & =
        \left( \binom{1}{1}\binom{2}{1, 1}(\xi^{0} + \xi^{0})^{1} \binom{1}{1}\binom{2}{2}(\xi^{0} + \xi^{0})^{1} \right)\\
            & = 6.
    \end{align}
    For $\pi \in \set{\pi^{(2)}, \pi^{(3)}, \pi^{(4)}}$ the computations are similar and yield in each case
    \begin{align}\label{ex:Lee_pi345}
        \sum_{t \in \compL} \left(\prod_{i = 0}^4 \binom{\pi_i}{t_{i0}, \ldots, t_{i4}}  \prod_{j = 1}^4\left(\xi^{-ij} + \xi^{ij}\right)^{t_{ij}}\right)
             =
        \left((\xi^{0} + \xi^{0}) +  (\xi^{-3} + \xi^{3}) + (\xi^{-3} + \xi^{3}) \right)
        = 0,
    \end{align}
    since $\sum_{i \in \ideal{3}} \xi^{i} = 0$.
    Putting the expressions given in \eqref{ex:Lee_pi1}, \eqref{ex:Lee_pi2} and \eqref{ex:Lee_pi345} in the right-hand side of \eqref{eq:MacWilliams_Lee} yields
    \begin{align}
        \frac{1}{\card{\code}} \sum_{\pi \in \decompLq} \krawL_{\pi}(\rho) \enumL{\rho}(\code) 
            = 
        \frac{1}{9} \left( 6 \cdot 1 + 6 \cdot 2 + 0\cdot 2 + 0\cdot 2 + 0\cdot 2 \right) = 2,
    \end{align}
    which, indeed, corresponds to $\enumL{\pi}(\dualcode) = 2$.
\end{example}

\subsection{Homogeneous Weight-Unit Partition}\label{app:example_Hom}
Let us now consider an example for the MacWilliams-type identities for the homogeneous weight-unit partition stated in Theorem \ref{thm:nonexistence_Macwill_Hom}.
\begin{example}\label{ex:homogeneous}
   Given $q=2$ and $s=4$, consider the code $\code \subset (\Zmod{16})^3$ generated by $G = \begin{pmatrix} 8 & 1 & 8\end{pmatrix}$. Notice, that there are four distinct homogeneous weight-unit decompositions occurring in $\code$, namely,
    \begin{gather}
        \pi^{(0)} = (3, 0, 0, 0), \quad \pi^{(1)} = (2, 0, 0, 1), \quad \pi^{(2)} = ( 2, 0, 1, 0), \quad \pi^{(3)}=(0,1,2,0),
    \end{gather}
    with the corresponding homogeneous weight-unit enumerators given by 
    \begin{align}
        \enumH{\pi^{(0)}}(\code) = 1, \quad  \enumH{\pi^{(1)}}(\code) = 6, \quad \enumH{\pi^{(2)}}(\code) = 1, \quad \enumH{\pi^{(3)}}(\code) = 8.
    \end{align}
    Note that the dual code of $\code$, $\dualcode$, is generated by
    \begin{align}
        G^{\perp} = \begin{pmatrix} 1 & 8 & 0 \\ 0 & 8 & 1 \end{pmatrix}.
    \end{align}
    Let $\rho = (1, 1, 1, 0) \in \mathbb{D}_{3}^{\mathsf{hom}}$ be a homogeneous weight-unit decomposition in $(\Zmod{16})^3$. By calculating the codewords of $\dualcode$, we observe that $\enumH{\rho}(\dualcode) = 16$.
    For every $\pi \in \set{\pi^{(0)}, \pi^{(1)}, \pi^{(2)},\pi^{(3)}}$ let us now compute the \kr coefficient $\krawH_{\pi}(\rho)$ as stated in Theorem \ref{thm:MacWilliams_Hom}.

    Let us start by choosing $\pi = \pi^{(0)}$. Notice  that there is only one composition $\rho$ with respect to $\pi^{(0)}$, i.e.,
    \begin{align}
        \mathrm{Comp}^{\mathsf{hom}}_{\pi^{(0)}}(\rho) = \Big\{(\underbrace{1, 1, 1, 0}_{t_Z}, \underbrace{0, 0, 0, 0}_{t_U}, \underbrace{0, 0, 0, 0}_{t_S}, \underbrace{0, 0, 0, 0}_{t_R})\Big\}.
    \end{align}
    Hence, we obtain
    \begin{align}
        \krawH_{\pi^{(0)}}(\rho) &= \prod_{i  \in\ZUSRset} \binom{\pi_i}{t_{iZ}, \ldots, t_{iR}}
        =
        \binom{\pi_Z}{t_Z}(q^{s-1}(q-1))^{t_{ZU}}\\
        &=
        \binom{3}{1, 1, 1, 0}8 = 48.
    \end{align}
    Now, let $\pi = \pi^{(1)}$.
    Note that the set of compositions $\compH$ of $\rho$ with respect to $\pi$ consists of only three elements, i.e.,
    \begin{align}
        \mathrm{Comp}^{\mathsf{hom}}_{\pi^{(1)}}(\rho) = \Big\{&t = (\underbrace{1, 1, 0, 0}_{t_Z}, \underbrace{0, 0, 0, 0}_{t_U}, \underbrace{0, 0, 0, 0}_{t_S}, \underbrace{0, 0, 1, 0}_{t_R}),\\  &s =(\underbrace{1, 0, 1, 0}_{s_Z}, \underbrace{0, 0, 0, 0}_{s_U}, \underbrace{0, 0, 0, 0}_{s_S}, \underbrace{0, 1, 0, 0}_{s_R})\\ &v =(\underbrace{0, 1, 1, 0}_{v_Z}, \underbrace{0, 0, 0, 0}_{v_U}, \underbrace{0, 0, 0, 0}_{v_S}, \underbrace{ 1, 0,0, 0}_{s_R})\Big\} .
    \end{align}
    Recalling that $q-1=1$ and omitting all the powers of $1$, we obtain 
    \begin{align}
      \krawH_{\pi^{(1)}}(\rho)  
        &= \binom{\pi_Z}{t_Z}\binom{\pi_R}{t_R} \left(q^{s-1}\right)^{t_{ZU}} + \binom{\pi_Z}{v_Z}\binom{\pi_R}{v_R}(q^{s-1})^{v_{ZU}} \\
        &=
        \binom{2}{1, 1,0, 0}\binom{1}{0, 0, 1, 0} 8 + \binom{2}{0, 1, 1, 0}\binom{1}{ 1, 0,  0, 0} 8\\
        &= 32.
    \end{align}
 Moreover, if $\pi = \pi^{(2)}$, we have
    \begin{align}
        \mathrm{Comp}^{\mathsf{hom}}_{\pi^{(2)}}(\rho) = \Big\{&t = (\underbrace{1, 1, 0, 0}_{t_Z}, \underbrace{0, 0, 0, 0}_{t_U}, \underbrace{0, 0, 1, 0}_{t_S}, \underbrace{0, 0, 0, 0}_{t_R}),\\  &s =(\underbrace{1, 0, 1, 0}_{s_Z}, \underbrace{0, 0, 0, 0}_{s_U}, \underbrace{0, 1, 0, 0}_{s_S}, \underbrace{0, 0, 0, 0}_{s_R})\\ &v =(\underbrace{0, 1, 1, 0}_{v_Z}, \underbrace{0, 0, 0, 0}_{v_U}, \underbrace{1, 0, 0, 0}_{v_S}, \underbrace{ 0, 0,0, 0}_{s_R})\Big\} .
    \end{align}
   It follows that
    \begin{align}
      \krawH_{\pi^{(2)}}(\rho)  &= \binom{\pi_Z}{t_Z}\binom{\pi_S}{t_S} \left(q^{s-1}\right)^{t_{ZU}} + \binom{\pi_Z}{s_Z}\binom{\pi_S}{s_S}(-1)^{s_{SU}} +   \binom{\pi_Z}{v_Z}\binom{\pi_S}{v_S}(q^{s-1})^{v_{ZU}} \\
        &=
        \binom{2}{1, 1,0, 0}\binom{1}{0, 0, 1, 0} 8 +\binom{2}{1, 0, 1, 0}\binom{1}{0,1,0,0}(-8)  +\binom{2}{0, 1, 1, 0}\binom{1}{ 1, 0,  0, 0} 8\\
        &= 16.
    \end{align}
    
    Lastly, for $\pi^{(3)}$, we have
    \begin{align}
        \mathrm{Comp}^{\mathsf{hom}}_{\pi^{(2)}}(\rho)
         = \Big\{&t = (\underbrace{0, 0, 0, 0}_{t_Z}, \underbrace{1, 0, 0, 0}_{t_U}, \underbrace{0, 1, 1, 0}_{t_S}, \underbrace{0, 0, 0, 0}_{t_R}),\\ &s =(\underbrace{0, 0, 0, 0}_{s_Z}, \underbrace{0, 1, 0, 0}_{s_U}, \underbrace{1, 0, 1, 0}_{s_S}, \underbrace{0, 0, 0, 0}_{s_R})\\ &v =(\underbrace{0, 0, 0, 0}_{v_Z}, \underbrace{0, 0, 1, 0}_{v_U}, \underbrace{1, 1, 0, 0}_{v_S}, \underbrace{ 0, 0,0, 0}_{s_R})\Big\} ,
    \end{align}
    which yields to
    \begin{align}
        \krawH_{\pi^{(2)}}(\rho) &=  \binom{\pi_U}{t_U}\binom{\pi_S}{t_S} \left(-q^{s-1}\right)^{t_{SU}} + \binom{\pi_U}{s_U}\binom{\pi_S}{s_S}0^{s_{UU}} +    \binom{\pi_U}{v_U}\binom{\pi_S}{v_S}(-1)^{v_{US}}(-q^{s-1})^{v_{ZU}} \\
        &=
        \binom{1}{1, 0, 0, 0} \binom{2}{0, 1,1, 0} (-8) + \binom{1}{0, 0, 1, 0}\binom{2}{ 1, 1,  0, 0} 8\\
        &= 0.
    \end{align}
    Finally, inserting these values in Equation \eqref{eq:MacWilliams_Hom}, yields
    \begin{align}
        \frac{1}{\card{\code}} \sum_{\pi \in \decompLq} \krawL_{\pi}(\rho) \enumL{\rho}(\code) 
        =\frac{1}{16} ( 48 \cdot 1 + 6 \cdot 32 + 1 \cdot 16+0 \cdot 8 ) = 16,
    \end{align} 
    which corresponds with $\enumH{\rho}(\dualcode)$.
\end{example}

\subsection{Subfield-Trace Partition}\label{app:example_Sub}
Lastly, we give an example to show the calculations of the MacWilliams-type identities of Corollary \ref{thm:MacWilliams_subfield}.
\begin{example}
    Given the finite field $\field_{27} = \field_3[\alpha]$ where $\alpha^3 = \alpha + 2$. Using the subfield-trace partition (as stated in \eqref{eq:subfield-trace_Fpr}) we partition $\field_{27}$ into the following four blocks
    \begin{gather}
        \IS_0 = \set{0},\quad
        \IS_1 = \field_{27}^{\times} = \set{1, 2}, \quad
        \IS_2 = \set{\alpha, \alpha + 1, \alpha + 2, 2\alpha, 2\alpha + 1, 2\alpha + 2}, \quad \\ \text{and }\quad
        \IS_3 = \field_{27} \setminus \left( \IS_0 \cup \IS_1 \cup \IS_2 \right)
    \end{gather}
    Consider the code $\code \subset \field_{27}^3$ generated by $G = \begin{pmatrix} 1 & \alpha & 1\end{pmatrix}$. By computing the codewords of $\code$ we observe that they can be partitioned into six distinct subfield-trace decompositions, namely $$(3, 0, 0, 0), (0, 2, 1,0), (0, 0, 2, 1), (0, 0, 1, 2), (0, 1, 0,2) \text{ and }(0, 0, 0, 3).$$ Each of the six subfield-trace decompositions are represented by $1, 2, 6, 4, 2$ and $12$ codewords, respectively.
    The dual $\dualcode$ of $\code$ has a generator matrix of the form
    \begin{align}
        G^{\perp} =
        \begin{pmatrix}
            1 & 0 & 2 \\
            0 & 1 & \alpha
        \end{pmatrix}
        .
    \end{align}
    It is easily checked that the vectors $(1, 0, 2)$ and $(2, 0, 1)$ are the only two codewords of $\dualcode$ of subfield-trace decomposition equal to $\rho = (1, 2, 0, 0)$. Thus, $\enumS{\rho}(\dualcode) = 2$.\\
    We now verify that indeed $\frac{1}{|\code|}\sum_{\pi \in \decompSq(\rho)} \krawS_{\pi}(\rho)\enumS{\pi}(\code) = 2$. We show detailed calculations of the \kr coefficient only for $\pi^{(0)} = (3, 0, 0, 0)$ and $\pi^{(1)} = (0, 2, 1, 0)$. Let us first assume, that $\pi = \pi^{(0)}$. Similar to the examples of the Lee and the homogeneous metric, we observe that
    \begin{align}
        \compS = \{(\underbrace{1, 2, 0, 0}_{\pi_0}, \underbrace{0, \ldots, 0}_{\pi_1}, \underbrace{0, \ldots, 0}_{\pi_2}, \underbrace{0, \ldots, 0}_{\pi_3})\}.
    \end{align}
    That is, for $t \in \compS$, we have $t_{00}= 1, t_{01} = 2$ and $t_{02} = \ldots = t_{33} = 0$. This yields
    \begin{align}
        \krawS_{\pi}(\rho) = \binom{3}{2, 1}(3-1)^2 = 12, \quad \text{and} \quad \krawS_{\pi}(\rho)\enumS{\pi}(\code) = 12\cdot 1 = 12.
    \end{align}
    Consider now $\pi = \pi^{(1)}$. For visual reasons, Figure \ref{fig:example_subfield} shows the set $\compS$ of compositions of $\rho$ with respect to $\pi$.
    \begin{figure}
        \centering
        \begin{tikzpicture}
            \draw (0,1) -- (6,1) -- (6,1.5) -- (0,1.5) -- cycle;
            \draw (2, 1) -- (2, 1.5);
            \draw[very thick] (4, 1) -- (4, 1.5);
            \foreach \i in {0, 1}{
                \node at (1 + 2*\i, 1.25) {\footnotesize $\IS_1$};
            }
            \node at (5, 1.25) {\footnotesize $\IS_2$};
            \draw[<->] (0, 1.6) -- (4, 1.6);
            \node[above] at (2, 1.6) {\footnotesize $\pi_1$};
            \draw[<->] (6, 1.6) -- (4, 1.6);
            \node[above] at (5, 1.6) {\footnotesize $\pi_{2}$};
            \node[right = 0.4] at (6, 1.25) {$\pi$};
            \draw (0,0) -- (6,0) -- (6,0.5) -- (0,0.5) -- cycle;
            \draw (0,-0.1) -- (6,-0.1) -- (6,-0.6) -- (0,-0.6) -- cycle;
            \foreach \i in {1, 2}{
                \draw (2, -1.2 + 0.6*\i) -- (2, -0.7 + 0.6*\i);
                \draw[very thick] (4, -1.2 + 0.6*\i) -- (4, -0.7 + 0.6*\i);
            }
            \node at (1, 0.25) {\footnotesize 0};
            \node at (3, 0.25) {\footnotesize $\IS_1$};
            \node at (5, 0.25) {\footnotesize $\IS_1$};
            \node at (1, -0.35) {\footnotesize $\IS_1$};
            \node at (5, -0.35) {\footnotesize $\IS_1$};
            \node at (3, -0.35) {\footnotesize $0$};
            \draw [decorate,decoration={brace,mirror}] (6.2,-0.7) -- (6.2,0.5) node [black,midway,right=0.2] {\footnotesize $\rho$};
        \end{tikzpicture}
        \caption{Illustration of $\compS$.}
        \label{fig:example_subfield}
    \end{figure}
    For the first option $t \in \compS$ shown in Figure \ref{fig:example_subfield}, we have $t_{01} = t_{11} = t_{21} = 1$ and all other $t_{ij}$'s are equal to $0$. The second option, i.e.,  $s \in \compS$, consists of $s_{11} = 2$ and $s_{20} = 1$ and $s_{ij} = 0$ for all other indices. Hence,
    \begin{gather}
        \krawS_{\pi}(\rho) = \binom{2}{1, 1}\binom{1}{1} (3-1)^{1+1} + \binom{2}{2}\binom{1}{1}(3-1)^2 = 12 \quad  \text{and}\quad \krawS_{\pi}(\rho) \enumS{\pi}(\code) = 24.
    \end{gather}
    The calculations of the remaining \kr coefficients are similar to the calculations of $\krawS_{\pi^{(0)}}(\rho)$ and $\krawS_{\pi^{(1)}}(\rho)$ leading to
    \begin{align}
        \frac{1}{|\code|}\sum_{\pi \in \decompSq(\rho)} \krawS_{\pi}(\rho)\enumS{\pi}(\code) = \frac{1}{27}\left( 12 + 24 + 0  - 12 - 6 + 36 \right) = \frac{1}{27}\cdot 54 = 2.
    \end{align}
\end{example}

\section*{Acknowledgments}
Jessica Bariffi was partly supported by the Federal Ministry of Education and Research of Germany in the program of "Souver\"an. Digital. Vernetzt." Joint project 6G-RIC, project identification number: 16KISK022. She is now funded by the European Union (DiDAX, 101115134). Views and opinions expressed are however those of the author(s) only and do not necessarily reflect those of the European Union or the European Research Council Executive Agency. Neither the European Union nor the granting authority can be held responsible for them.
Giulia Cavicchioni is member of the INdAM Research Group GNSAGA. She  conducted the majority of this research at the University of Trento, Trento, Italy, and is currently affiliated with the Institute of Communication
and Navigation, German Aerospace Center, Oberpfaffenhofen, Germany.
Violetta Weger is supported by the European Union's Horizon 2020 research and innovation programme under the Marie Sk\l{}odowska-Curie grant agreement no. 899987.

\bibliographystyle{elsarticle-num}
\bibliography{biblio}

\end{document}